%% file: CAP_arxiv.tex
\documentclass[ssy,preprint]{imsart}

\RequirePackage[OT1]{fontenc}
\RequirePackage{amsthm,amsmath}
\RequirePackage[numbers]{natbib}
\RequirePackage[colorlinks,citecolor=blue,urlcolor=blue]{hyperref}

\startlocaldefs
\numberwithin{equation}{section}
\theoremstyle{plain}

\endlocaldefs

\usepackage{amsfonts,amssymb,tikz,pgfplots,comment,empheq}

\newcommand{\pr}{\mathbb P}
\newcommand{\ep}{\mathbb E}

\newcommand{\rew}[2]{\mathbb E_{#1}\left[T_{#2}^{A}\right]}
\newcommand{\rep}[3]{\mathbb E_{({#3},{#1})}\left[T_{({#3},{#2})}^{P_{#3}}\right]}

\newcommand{\alp}[3]{\alpha_{#1}\langle{#2};{#3}\rangle}
\DeclareMathOperator{\ed}{Exponential}
\newtheorem{theorem}{Theorem}
\newtheorem{lemma}{Lemma}

\newtheorem*{theorem:eij}{Theorem \ref{thm:eij}}

\begin{document}
\begin{frontmatter}
\title{Clearing Analysis on Phases:\\ Exact Limiting Probabilities for Skip-Free, Unidirectional, Quasi-Birth-Death Processes}
\runtitle{Clearing Analysis on Phases}

\begin{aug}
\author{\fnms{Sherwin} \snm{Doroudi}\thanksref{m1},
\ead[label=e1]{sdoroudi@andrew.cmu.edu}}
\author{\fnms{Brian} \snm{Fralix}\thanksref{t2,m2},
\ead[label=e2]{bfralix@clemson.edu}}\\
\and
\author{\fnms{Mor} \snm{Harchol-Balter}\thanksref{t3,m1}
\ead[label=e3]{harchol@cs.cmu.edu}
\ead[label=u1,url]{http://www.foo.com}
}

\affiliation{Carnegie Mellon University\thanksmark{m1} and Clemson University\thanksmark{m2}}

\thankstext{t2}{Partially supported by NSF grant CMMI-1435261.}
\thankstext{t3}{Funded by NSF-CMMI-1334194 and NSF-CSR-1116282, by the Intel Science and Technology Center for Cloud Computing, and by a Google Faculty Research Award 2015/16.}
\runauthor{S. Doroudi, B. Fralix, M. Harchol-Balter}


\printead{e1,e2,e3}\\

\end{aug}

\begin{abstract}
Many problems in computing, service, and manufacturing systems can be modeled via infinite repeating Markov chains with an infinite number of levels and a finite number of phases.
Many such chains are quasi-birth-death processes (QBDs) with transitions that are \emph{skip-free in level}, in that one can only transition between consecutive levels, and \emph{unidirectional in phase}, in that one can only transition from lower-numbered phases to higher-numbered phases.   We present a procedure, which we call Clearing Analysis on Phases (CAP), for determining the limiting probabilities of such Markov chains exactly.  The CAP method yields the limiting probability of each state in the repeating portion of the chain as a linear combination of \emph{scalar bases} raised to a power corresponding to the level of the state.  The weights in these linear combinations can be determined by solving a finite system of linear equations.
\end{abstract}



\end{frontmatter}

\section{Introduction}

This paper studies the stationary distribution of Class $\mathbb M$ Markov chains, which are continuous time Markov chains (CTMCs)\footnote{The methodology presented in this paper can easily be modified to apply to discrete time Markov chains.} having the following properties (see Fig. \ref{fig:chain} and Fig. \ref{fig:detail}):
\begin{itemize}

\item The Markov chain has a state space, $\mathcal E$, that can be decomposed as $\mathcal E = \mathcal{R} \cup \mathcal{N}$, where $\mathcal R$ represents the infinite \emph{repeating portion} of the chain, and $\mathcal{N}$ represents the finite \emph{nonrepeating portion} of the chain.\footnote{We note that this partition is not unique.}

\item The repeating portion is given by
\begin{align*} \mathcal{R} &\equiv \{(m,j):  0 \leq m \leq M, j \geq j_{0}\} \end{align*} where both $M$ and $j_{0}$ are finite nonnegative integers.  We refer to a state $(m,j) \in \mathcal{R}$ as currently being in phase $m$ and level $j$.  For each $j \geq j_{0}$, level $j$ is given by
\begin{align*} L_{j} \equiv \{(0,j), (1,j), \ldots, (M,j)\}. \end{align*}
Throughout this paper, we index phases by $i$, $k$, $m$, and $u$, and we index levels by $j$ and $\ell$.

\item Transitions between a pair of states in $\mathcal{N}$ may exist with any rate.

\item Transitions from states in $\mathcal{N}$ to states in $\mathcal{R}$ may only go into states in $L_{j_{0}}$, but may exist with any rate.

\item Transitions from states in $\mathcal{R}$ to states in $\mathcal{N}$ may only come from states in $L_{j_{0}}$, but may exist with any rate.

\item Transitions between two states in $\mathcal{R}$ that are both in the same phase, $m$, (e.g., the ``horizontal'' transitions in Fig. \ref{fig:chain} and Fig. \ref{fig:detail}) are described as follows, with $q(x,y)$ denoting the transition rate from state $x$ to state $y$:
\begin{align*}
\lambda_{m}&\equiv q((m,j), (m,j+1))& (&0 \leq m \leq M,\ \  j \geq j_{0}) \\
\mu_{m}&\equiv q((m,j), (m,j-1))& (&0 \leq m \leq M,\ \  j \geq j_{0} + 1).
\end{align*}

\item We express transition rates between two states in $\mathcal{R}$, which transition out of a state in phase $m$ to a state in another phase (e.g., the ``vertical'' transitions in Fig. \ref{fig:chain} and the ``vertical'' and ``diagonal'' transitions in Fig. \ref{fig:detail}) using the notation $\alp m{\Delta_1}{\Delta_2}$, where $\Delta_1\ge 1$ is the increase in phase from $m$ to $m+\Delta_1$ (i.e., the ``vertical'' shift) and $\Delta_2\in\{-1,0,1\}$ is the change in level, if any, from $j$ to $j+\Delta_2$ (i.e., the ``horizontal'' shift).  Note that $\Delta_1\ge1$ indicates that only transitions to higher-numbered phases are allowed, while $\Delta_2\in\{-1,0,1\}$ indicates that each transition may change the level by at most 1 in either direction. More specifically, these transitions are described as follows:
\begin{align*}
\alp m{i-m}{-1}&\equiv q((m,j), (i,j-1))& (&0 \leq m < i \leq M,\ \  j \geq j_{0}+1) \\
\alp m{i-m}0&\equiv q((m,j), (i,j)) & (&0 \leq m < i \leq M,\ \  j \geq j_{0}) \\
\alp m{i-m}{1}&\equiv q((m,j), (i, j+1)) & (&0 \leq m < i \leq M,\ \  j \geq j_{0}).
\end{align*}
We will also use the shorthand notation
\begin{align*} \alpha_{m} 
&=\sum_{i={m+1}}^{M}(\alp m{i-m}{-1}+ \alp m{i-m}{0} +\alp m{i-m}1)\end{align*} throughout the paper to represent the \emph{total outgoing transition rate to other phases} from states in phase $m$ with level $j\ge j_0+1$.

\item The Markov chain must be ergodic.

\end{itemize}

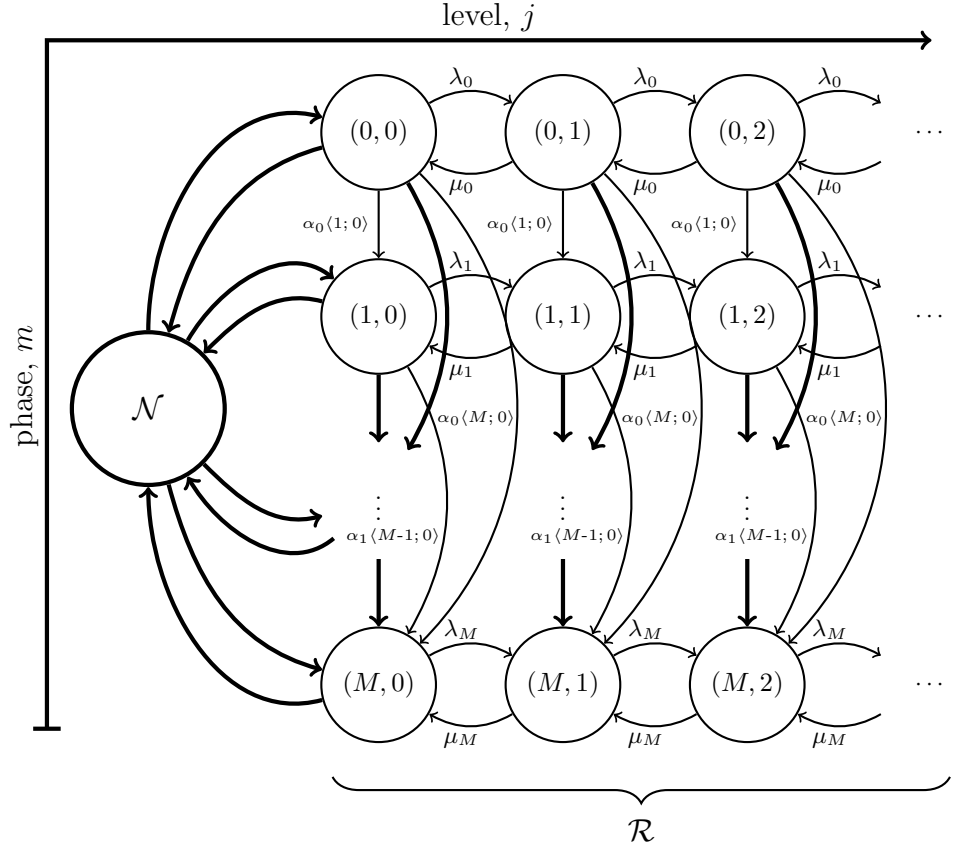
\begin{figure}
\begin{center}
\input{chain.tikz}
\end{center}
\caption{The structure of class $\mathbb M$ Markov chains.  In this case $j_0=0$ and, for simplicity, $\alp{m}{i-m}{\pm1}=0$. The chain is made up of a non-repeating portion, $\mathcal N$ (shown here as an aggregation of states), and a repeating portion, $\mathcal R$.  Within $\mathcal R$, each phase, $m$, corresponds to a ``row'' of states, and each level, $j$, corresponds to a ``column'' of states.  Transitions between levels in each phase of the repeating portion, $\mathcal R$, are \emph{skip-free}: all such transitions move only \emph{one step} to the ``left'' or ``right.''  Transitions between phases in each level of $\mathcal R$ are \emph{unidirectional}: all such transitions move ``downward.''  The thicker arrows denote \emph{sets} of transitions (transitions rates for these sets are omitted from the figure).}
\label{fig:chain}
\end{figure}

Markov chains in class $\mathbb M$  are examples of quasi-birth-death processes (QBDs), with increments and decrements in level corresponding to ``births'' and ``deaths,'' respectively.  We say that transitions in class $\mathbb M$ chains are \emph{skip-free in level}, in that the chain does not allow for the level to increase or decrease by more than 1 in a single transition.  We also say that transitions in class $\mathbb M$ chains are \emph{unidirectional in phase}, in that transitions may only be made to states having either the same phase or a higher phase in the repeating portion.  Note however that phases may be skipped:  for example, transitions from a state in phase $2$ to a state in phase $5$ may exist with nonzero rate.

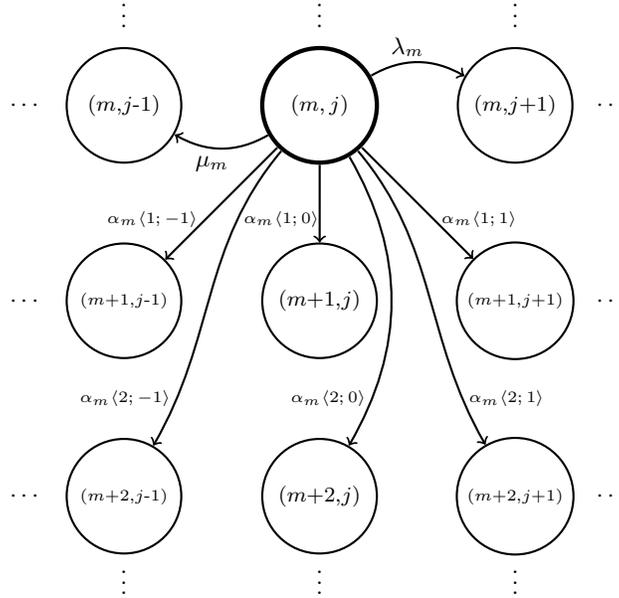
\begin{figure}
\begin{center}
\input{detail.tikz}
\end{center}
\caption{Another more detailed look at the transition structure of class $\mathbb M$ Markov chains.  For simplicity, only the set of transitions that are possible \emph{from} state $(m,j)$ (where $j\ge j_0+1$) to states in phases $m$, $m+1$, and $m+2$ are shown.  
Note that all transitions from $(m,j)$ are either to the left, to the right, or downward.  Furthermore, all transitions can decrease or increase the level by at most one.}
\label{fig:detail}
\end{figure}

Many common queueing systems arising in computing, service, and manufacturing systems can be modeled with CTMCs from class $\mathbb{M}$.  For such systems, one often needs to track both the number of jobs in the system and the state of the server(s), where each server may be in one of several states, e.g., working, fatigued, on vacation, etc.  When modeling a system with a class $\mathbb{M}$ Markov chain, we often use the level, $j$, of a state $(m,j)$ to track the number of jobs in the system, and we use the phase, $m$, to track the state of the server(s) and/or the arrival process.  For example, a change in phase could correspond to (i) a policy modification that results in admitting more customers, as captured by an increase in ``arrival rate'' from $\lambda_{m}$ to $\lambda_{i}$, where $\lambda_{i} > \lambda_{m}$ or (ii) a change in the state of the servers leading to an increase or decrease in the service rate from $\mu_m$ to $\mu_i$.  A few examples of systems that can be modeled by Class $\mathbb{M}$ Markov chains are presented in Section \ref{sec:examples}.

\subsection{The matrix-geometric approach}\label{sec:matrixgeo}

One way of studying the stationary distribution, $\pi$, of a class $\mathbb{M}$ Markov chain is to observe that it exhibits a matrix-geometric structure on $\mathcal{R}$. More specifically, if we let $\vec\pi_{j}$ represent the limiting probability of the states in $L_j$, that is, $\vec\pi_{j} \equiv (\pi_{(0,j)}, \pi_{(1,j)}, \ldots, \pi_{(M,j)})$, then for $j \geq j_{0}$
\begin{align*} \vec\pi_{j+1} &= \vec\pi_{j}\mathbf{R} \end{align*} where $\mathbf{R} \in \mathbb{R}^{(M+1) \times (M+1)}$ is referred to as the \emph{rate matrix} associated with the chain.  If we let the \emph{sojourn rate} of state $x$ be defined by \[\nu_{x}=\sum_{y\neq x}q(x,y),\] then we can describe the elements of $\mathbf R$ probabilistically as follows: the element, $R_{i,m}$, in row $i$, column $m$ of $\mathbf{R}$ can be interpreted as $\nu_{(i,j)}$ times the expected cumulative amount of time the chain spends in state $(m,j+1)$ before making a transition into a level strictly below $j+1$, given the chain starts in state $(i,j)$.  For most QBDs, one cannot derive an exact expression for each element of $\mathbf{R}$, but there are many ways to compute an approximation of $\textbf{R}$ numerically:  see for example \cite{latouche1993logarithmic,bright1995calculating}.  Readers interested in further details should consult the matrix-analytic texts of Neuts \cite{neuts1981matrix}, Latouche and Ramaswami \cite{mbook}, and He \cite{he2014fundamentals}.  Queueing textbooks of a broader scope that also discuss matrix-analytic methods include Asmussen \cite{asmussen2003applied} and Harchol-Balter \cite{harchol2013performance}.  Once $\mathbf{R}$---or good approximations for $\mathbf{R}$---have been found, then $\vec\pi_{j} = \vec\pi_{j_{0}}\textbf{R}^{j - j_0}$ for $j \geq j_{0}$, and so all remaining limiting probabilities, $\pi_{x}$, for $x \in \mathcal{N}$, can be found using the balance equations and the normalization constraint.

There are many examples of QBDs with a rate matrix, $\textbf{R}$, that can be computed exactly through a finite number of operations.  One class of QBDs having a closed-form rate matrix is presented in Ramaswami and Latouche \cite{ramaswami1986general}, with an extension to Markov chains of $GI/M/1$-type given in Liu and Zhao \cite{liu1996determination}.  Other classes of QBDs having explicitly computable rate matrices are considered in the work of van Leeuwaarden and Winands \cite{johan06} and van Leeuwaarden et al. \cite{johan09}, with both of these studies being much closer to our work, since most (but not all) of the types of Markov chains studied in \cite{johan06}, and all of the chains discussed in \cite{johan09} belong to class $\mathbb{M}$.  In \cite{johan06, johan09} combinatorial techniques are used to derive expressions for each element of $\textbf{R}$ that can be computed exactly after a finite number of operations, but their methods are not directly applicable to all class $\mathbb{M}$ Markov chains as they further assume that $\lambda_{m}$ and $\mu_{m}$ are the same for $0 \leq m \leq M-1$, and they also assume that for each $0 \leq m \leq M-1$, any transitions leaving phase $m$ must next enter phase $m+1$ (i.e., they assume \emph{phase} transitions are \emph{skip-free}---in addition to being unidirectional---within the repeating portion of the chain).

Even closer to our work is the work of Van Houdt and van Leeuwaarden \cite{johan11}, which presents an approach for the explicit calculation of the rate matrix for a broad class of QBDs including those in class $\mathbb M$. This approach involves solving higher order (scalar) polynomial equations, the solutions to which are expressed as infinite sums, which typically cannot be computed in closed-form.  However, \cite{johan11} also gives an approach for calculating closed-form rate matrices for a class of Markov chains called tree-like QBDs.  Tree-like QBDs neither contain nor are contained by class $\mathbb M$, although there is significant overlap between the two.  Transitions between phases (within a level) in tree-like QBDs form a \emph{directed tree}, while transitions between phases in class $\mathbb M$ Markov chains form a \emph{directed acyclic graph}.  Specifically, unlike class $\mathbb M$ chains, tree-like QBDs \emph{do not} allow for a pair of phases $i\neq k$ to both have transitions to the same phase $m$ (i.e., tree-like QBDs do not allow for both $\alp i{m-i}{\Delta}>0$ and $\alp k{m-k}{\Delta'}>0$ when $i\neq k$ and $\Delta,\Delta'\in\{-1,0,1\}$).

 \subsection{Our approach: Clearing Analysis on Phases (CAP)}

In this study we introduce the CAP (Clearing Analysis of Phases) method for evaluating the stationary distribution of  class $\mathbb{M}$ Markov chains.  This method proceeds iteratively among the \emph{phases}, by first expressing all $\pi_{(0,j)}$ probabilities, for $j \geq j_{0}$, in terms of $\pi_{x}$ probabilities for $x \in \mathcal{N}$.  Once each element $\pi_{(m,j)}$ for a fixed phase $m$, $j \geq j_{0}$ has been expressed in terms of $\{\pi_{x}\}_{x \in \mathcal{N}}$, we then do the same for all $\pi_{(m+1,j)}$ terms.  After each $\pi_{(M,j)}$ expression has been determined, we use the balance equations and normalization constraint to solve for the remaining $\{\pi_{x}\}_{x \in \mathcal{N}}$ probabilities.  CAP takes its name from the fact that, between two phase transitions, class $\mathbb{M}$ Markov chains behave like an M/M/1/clearing model, that is, each phase is likened to a birth-death process that experiences ``clearing'' or catastrophic events in accordance to an independent Poisson process.  In our model, these ``clearings'' corresponds to a change in phase.

One major advantage of the CAP method is that it avoids the task of finding the complete rate matrix, $\textbf{R}$, entirely, while yielding expressions for $\pi_{(m,j)}$ that only involve raising $M+1$ {\em scalars} to higher powers.  There exists one such scalar, $r_m$, for each phase, $m\in\{0,1,\ldots,M\}$. These scalars, referred to throughout as \emph{base terms}, are actually the diagonal elements of the rate matrix, $\textbf{R}$, i.e.,
\begin{align*}
r_{m} = R_{m,m},\qquad(0 \leq m \leq M)
\end{align*}
and the transition structure of class $\mathbb{M}$ Markov chains makes these elements much easier to compute than any of the other nonzero elements of $\textbf{R}$.  Furthermore, the structure of $\pi_{(m,j)}$ depends entirely on the number of base terms that agree with one another.  For example, when all nonzero base terms are distinct, one can show that
\begin{align}\label{eq:solution_form}
\pi_{(m,j)}=\sum_{k=0}^m c_{m,k}r_k^{j-j_0},
\end{align}
for $0 \leq m \leq M$, where the $\{c_{m,k}\}_{0 \le k\le m\le M}$ values are constants that do not vary with $j$, and can be computed exactly by solving a linear system of $O(M^2+|\mathcal N|)$ linear equations.

In the case where all base terms agree, we instead find that
\begin{align}\label{eq:solution_form2}
\pi_{(m,j)} = \sum_{k=0}^{m}c_{m,k}{j - (j_{0} + 1) + k \choose k}r_{0}^{j - j_{0}},
\end{align}
where again, the $c_{m,k}$ terms can be computed by solving a linear system.

In retrospect, it is of no surprise that $\pi_{(m,j)}$ can be expressed as a linear combination of scalars, each raised to the power of $j-j_0$, as in Equations \eqref{eq:solution_form} and \eqref{eq:solution_form2}:  $\mathbf{R}$ must be upper-triangular for class $\mathbb M$ chains. This follows by observing that $R_{i,m}$ is $\nu_{(i,j_0)}$ times the expected cumulative amount of time spent in state $(m, j_{0} + 1)$ before returning to  $L_{j_{0}}$, given initial state $(i,j_{0})$, and this value is 0 when $i>m$.  Since $\mathbf{R}$ is upper-triangular, its eigenvalues are simply its diagonal elements, which are also the diagonal elements of the Jordan normal form of $\mathbf{R}$---see e.g., Chapter 3 of Horn and Johnson \cite{horn2012matrix}---from which we know that $\pi_{(m,j)}$ can be expressed as a linear combination of scalars, each raised to the power of $j-j_0$.  Although in theory, our solution form could be recovered by first computing $\mathbf{R}$ and then numerically determining $\mathbf{R}$ in Jordan normal form, such a procedure is often inadvisable.  The structure of the Jordan normal form of a matrix can be extremely sensitive to small changes in one or more of its elements, particularly when some of its eigenvalues have algebraic multiplicity larger than one, as is the case for all of the models discussed in \cite{johan06, johan09}.  Fortunately, the CAP method can handle these cases as well with little additional difficulty.

The statement and proofs of this paper's main results are presented in Section \ref{sec:main}.  This proof relies on some results regarding M/M/1/clearing models; the proofs of these results are deferred to Section \ref{sec:clearing}.  In Section \ref{sec:scope} we briefly touch upon how the CAP method may be applied to chains beyond those in class $\mathbb M$.

\subsection{Recursive Renewal Reward, ETAQA, and other techniques}

We brief\-ly review existing techniques for solving QBDs beyond the matrix-geometric approach and comment on their connection to the CAP method.

Gandhi et al. \cite{gandhi2013exact,gandhi2014exact} use renewal theory to determine exact mean values and $z$-transforms of various metrics for a subclass of chains in $\mathbb M$ via the Recursive Renewal Reward (RRR) method.  The class of chains they study do not allow for ``diagonal'' transitions (i.e., $\alp m{i-m}{\pm1}=0$). Unlike our method, RRR \emph{cannot} be used to determine a formula for a chain's limiting probability distribution in finitely many operations. 
While there is overlapping intuition and flavor between CAP and RRR---both methods make use of renewal reward theory---CAP is \emph{not} an extension of RRR and does not rely on any of the results from \cite{gandhi2013exact,gandhi2014exact}.

The Efficient Technique for the Analysis of QBD-processes by Aggregation (ETAQA), first proposed by Ciardo and Simirni \cite{etaqa1}, combines ideas from matrix analytic and state aggregation approaches in order to compute various exact values (e.g., mean queue length) for a wide class of Markov chains.  By design, ETAQA yields the limiting probability of the states in the non-repeating portion, $\mathcal N$, along with the limiting probabilities of the states in the first level (or first few levels) of the repeating portion, $\mathcal R$.  The limiting probabilities of the remaining states (i.e., higher level states) are aggregated, which allows for the speedy computation of exact mean values and higher moments of various metrics of interest.  In particular, ETAQA involves solving a system of only $O(|\mathcal N|+M)$ linear equations.  Although originally applicable to a narrow class of chains (see \cite{etaqa1,etaqa2} for details), ETAQA can be generalized so as to be applicable to M/G/1-type, GI/M/1-type, and QBD Markov chains, including those in class $\mathbb M$ (see the work of Riska and Smirni \cite{etaqa4,etaqa5}).  Stathopoulos et al. \cite{etaqa3} show that ETAQA is also well suited for numerical computations; ETAQA can be adapted to avoid the numerical problems alluded to in Section \ref{sec:matrixgeo}.  Unlike the CAP method, ETAQA (like RRR) \emph{cannot} be used to determine a formula for a chain's limiting probability distribution (across all states) in finitely many operations.

For certain class $\mathbb M$ Markov chains, one can also manipulate generating functions to derive limiting probabilities, such as in the work of Levy and Yechiali \cite{levy76} and the work of Phung-Duc \cite{phung2014exact}, where this type of approach is used to solve multi-server vacation and setup models, respectively.  This approach is covered in greater generality in a technical report by Adan and Resing \cite{adan99}.  We note that although generating function approaches can yield solutions of a form similar to those found using the CAP method, the two approaches differ in methodology.

\section{Examples of class {$\mathbb M$} Markov chains}\label{sec:examples}

In this section we provide several examples of queueing systems which can be modeled by class $\mathbb M$ Markov chains.  In each example we will use the phase, $m\in\{0,1,\ldots,M\}$, to track the ``state'' of the server(s) and/or the arrival process, and the level, $j$, to track the number of jobs in the system.  Of course, there are many systems beyond those covered in this section that can be modeled by class $\mathbb M$ Markov chains.  For example, class $\mathbb M$ chains were recently used to model medical service systems in \cite{delasay2013modeling}, \cite{chan2014impact}, and \cite{selen2015snowball}.

\subsection{Single server in different power states}\label{sec:sleep}

\begin{figure}
\begin{center}
\input{sleep.tikz}
\end{center}
\caption{The Markov chain for a single server in different power states.  State $(m,j)$ indicates that the server is in state $m$ (0=off, 1=sleep, 2=on) with $j$ jobs in the system.}
\label{fig:sleep}
\end{figure}
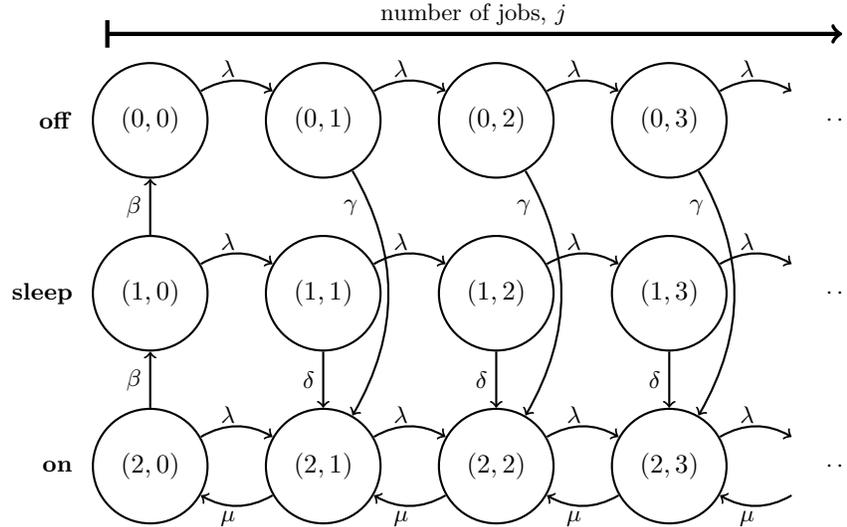

Consider a computer server that can be in one of three different power states: on, off, or sleep.  In the \textbf{on} state, the server is fully powered and jobs are processed at rate $\mu$.  In the \textbf{off} state, the server consumes no power, but jobs cannot be processed.  When the server is idle, it is desirable to switch to the off state in order to conserve power, however there is a long setup time, distributed $\ed(\gamma)$, needed to turn the server back on when work arrives.  Because of this setup time, it is common to switch to a state called the \textbf{sleep} state, where the server consumes less power than the \textbf{on} state, but where there is a shorter setup time, distributed $\ed(\delta)$, for turning the server on.  It is also common to purposefully impose a waiting period, distributed $\ed(\beta)$, in powering down a server (from on to sleep, and again from sleep to off) once it is idle, which is useful just in case new jobs arrive soon after the server becomes idle.  See \cite{gandhi2012autoscale} for more details.

Fig. \ref{fig:sleep} shows a Markov chain representing this setting.  This is a class $\mathbb M$ chain with $M+1=3$ phases: \textbf{off} ($m=0$), \textbf{sleep} ($m=1$), and \textbf{on} ($m=2$).  For this chain, $j_0=1$ and the non-repeating portion of the state space is $\mathcal N=\{(0,0),(1,0),(2,0)\}$, while $\lambda_0=\lambda_1=\lambda_2=\lambda$, $\mu_0=\mu_1=0$, $\mu_2=\mu$, $\alp020=\gamma$, and $\alp110=\delta>\gamma$ (all other $\alp m{m-i}\Delta$ transition rates are zero).

The system becomes much more interesting when there are multiple servers, where each can be in one of the above 3 states. In the case of 2 servers, there will be $6$ phases, corresponding to: (off,off), (off,sleep), (off,on), (sleep,sleep), (sleep,on), (on,on).  Note than in this case, phase transitions will include transitions with rates $2\gamma$, $\gamma+\delta$, and $2\delta$, as both servers may be attempting to turn on at the same time.
In general, a system with $a$ servers and $b$ server states will have $\binom{a+b-1}{a}$ phases.

\subsection{Server fatigue}

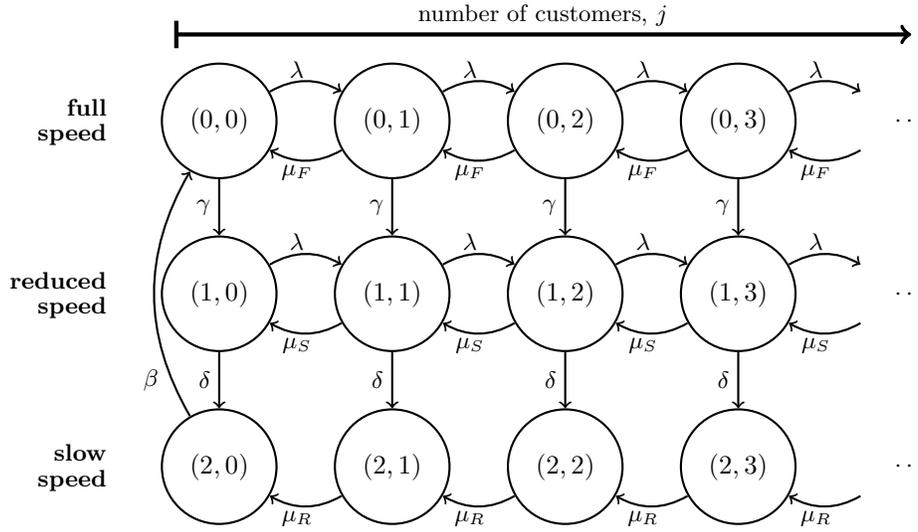
\begin{figure}
\begin{center}
\input{fatigue.tikz}
\end{center}
\caption{The Markov chain for a server susceptible to fatigue.  State $(m,j)$ indicates server state $m$ (0=full speed, 1=reduced speed, 2=slow speed) with $j$ customers in the system.}
\label{fig:fatigue}
\end{figure}



Consider a human server who starts her shift full of energy and works quickly (at rate $\mu_F$). As time passes and fatigue sets in, she gets slower and slower (first she slows down to a reduced rate $\mu_R$ and eventually to a very slow rate $\mu_S$, where $\mu_S < \mu_R < \mu_F$). At some point it makes sense to replace her with a fresh human server. However, before we can do that, she needs to finish serving her queue of existing customers, while no longer accepting further arrivals. We assume that the time it takes for the new replacement to start working is distributed $\ed(\beta)$.

Fig. \ref{fig:fatigue} shows a Markov chain representing this setting.  This is a class $\mathbb M$ chain with $M+1=3$ phases: \textbf{full speed} ($m=0$), \textbf{reduced speed} ($m=1$), and \textbf{slow speed} ($m=2$).  For this chain, $j_0=1$ and the non-repeating portion of the state space is $\mathcal N=\{(0,0),(1,0),(2,0)\}$, while $\lambda_0=\lambda_1=\lambda$, $\lambda_2=0$, $\mu_0=\mu_F$, $\mu_1=\mu_R<\mu_F$, $\mu_2=\mu_S<\mu_R$, $\alp010=\gamma$ and $\alp110=\delta$ (all other $\alp m{m-i}\Delta$ transition rates are zero).

Again, the system becomes much more interesting when there are multiple servers, where each can be in one of the above 3 states.

\subsection{Server with virus infections}

\begin{figure}
\begin{center}
\input{detection.tikz}
\end{center}
\caption{The Markov chain for a server vulnerable to viruses.  State $(m,j)$ indicates server state $m$ (0=uninfected, 1=undetected infection, 2=detected infection) with $j$ jobs in the system.}
\label{fig:detection}
\end{figure}
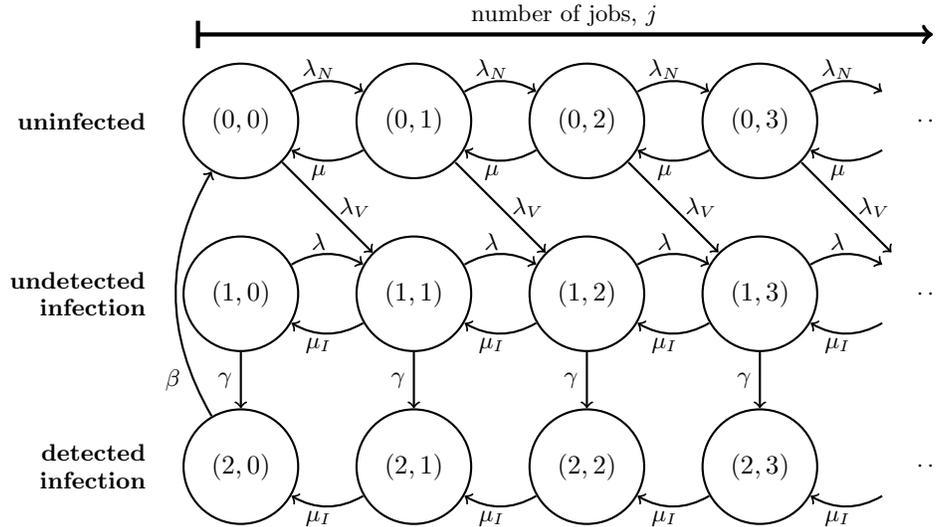

Imagine a computer server that is vulnerable to viruses.  We present a stylized model where normally, the server is \textbf{uninfected} and receives jobs with rate $\lambda$ and processes them with rate $\mu$. While most jobs are normal (i.e., not virus carriers), arriving at rate $\lambda_N$, every once in a while, one of the arriving jobs brings with it a virus, with rate $\lambda_V=\lambda-\lambda_N$.  The virus causes the server to become \textbf{infected}, reducing the server's service rate from $\mu$ to $\mu_I$. It takes a duration of time distributed $\ed(\gamma)$  for the server to detect that it is infected.  Once the infection is \textbf{detected}, the server stops accepting new jobs, and once all remaining jobs are processed, the server is able to use antivirus software to remove the virus in a duration of time distributed $\ed(\beta)$.  Once the virus is removed, the server is again uninfected and will resume accepting jobs, processing them at a restored service rate of $\mu$.  We model a single server as being in one of 3 states, each of which will make up a \emph{phase} of our Markov chain: \textbf{uninfected} ($m=0$), \textbf{undetected infection} ($m=1$), and \textbf{detected infection} ($m=2$).

Fig. \ref{fig:detection} shows a class $\mathbb M$ Markov that represents this setting.  
For this chain, $M=2$, $j_0=1$, $\mathcal N=\{(0,0),(1,0),(2,0)\}$, $\lambda_0=\lambda_N$, $\lambda_1=\lambda=\lambda_N+\lambda_V$, $\lambda_2=0$, $\mu_0=\mu$, $\mu_1=\mu_2=\mu_I$, $\alp011=\lambda_V$, and $\alp110=0$ (all other $\alp m{m-i}\Delta$ transition rates are zero).  

\section{Results}\label{sec:main}

In this section we first present a key theorem from the literature that enables the CAP method (Theorem \ref{thm:main}).  We then introduce some preliminary notation, and an original result, Theorem \ref{thm:eij}.  Finally, we present the main results of our paper, Theorems \ref{THM:RESULT}, \ref{THM:RESULT2}, and \ref{THM:RESULT3}, the proofs of which will depend on both Theorems \ref{thm:main} and \ref{thm:eij}.

\subsection{A key idea}

Consider an ergodic CTMC with state space, $S$, and consider a nonempty proper subset, $A\subsetneq S$, with states $x,z\in A$.  The CAP method involves calculating quantities of the form
\begin{align*}
\rew zx&\equiv\ep\left[\substack{
\mbox{cumulative time spent in state $x$ until next}\\
\mbox{transition leaving set $A$, given initial state $z$}
}\right]
\end{align*}
in order to determine the limiting probabilities of the Markov chain of interest. Theorem \ref{thm:main} (from Theorem 5.5.1 of \cite{mbook}) gives an expression for the limiting probabilities of the Markov chain in terms of the quantities $\rew zx$.
\begin{theorem}\label{thm:main}
Suppose $A \subsetneq S$.  Then for each $x \in A$, the limiting probability of being in state $x$, $\pi_x$, can be expressed as
\[\pi_x=\sum_{y\in A^c}\sum_{z\in A}\pi_y q(y,z)\rew zx,\] where $q(y,z)$ is the transition rate from state $y$ to state $z$ and $A^c\equiv S\backslash A$.
\end{theorem}
\begin{proof}
See Theorem 5.5.1 of \cite{mbook}.
\end{proof}
Intuitively, we are expressing the long run fraction of time that we reside in state $x$, $\pi_x$, as a weighted average of the cumulative time spent in state $x$ during uninterrupted visits to states in $A$, $\rew zx$, conditioned on the choice of state, $z\in A$, by which we enter $A$.  The weights in this average represent the rate at which visits to $A$ via $z$ occur, which involves conditioning on the states $y\in A^c$ by which one may transition to $z\in A$.  We illustrate $S$, $A$, $y$, $z$, and $x$ in Fig. \ref{fig:xyz}.

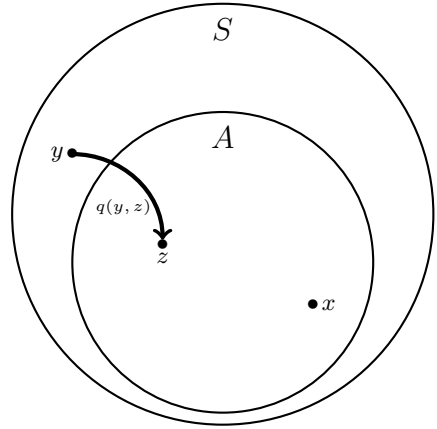
\begin{figure}
\begin{center}
\input{xyz.tikz}
\end{center}
\caption{For any $x\in A$, Theorem \ref{thm:main} gives $\pi_x$ as a linear combination of quantities $\rew zx$ by conditioning on the states $y\in A^c$, by which one may transition to states $z\in A$.  This figure shows one such $(y,z)$ pair.}
\label{fig:xyz}
\end{figure}

As an example, consider the simple case where $A=\{x\}$.  In this case, Theorem \ref{thm:main} yields
\begin{align*}
\pi_x&=\sum_{y\in A^c}\sum_{z\in A}\pi_y q(y,z)\rew zx=\sum_{y\neq x}\pi_y q(y,x)\mathbb E_x\left[T_x^{\{x\}}\right]=\frac{\sum_{y\neq x}\pi_y q(y,x)}{\sum_{y\neq x} q(x,y)},
\end{align*}
and so \[\pi_x\sum_{y\neq x} q(x,y)=\sum_{y\neq x}\pi_y q(y,x),\] which is simply the balance equation associated with state $x$.

Theorem \ref{thm:main} is used in \cite{mbook} to establish the matrix-geometric structure of the stationary distribution of QBD chains.  The same argument can be used to establish the matrix-geometric structure satisfied by the stationary distribution, $\pi$, of a class $\mathbb{M}$ Markov chain on $\mathcal{R}$.  Fix a level $j \geq j_{0}$, and define $A = \bigcup_{\ell \geq j+1}L_{\ell}$.  Then for each state $(m,j+1) \in L_{j+1}$, we have
\begin{align*}
\pi_{(m,j+1)} &= \sum_{i=0}^{M}\sum_{k=0}^{M}\pi_{(i,j)}q((i,j), (k,j+1))\mathbb{E}_{(k,j+1)}\left[T_{(m,j+1)}^{A}\right] \\
&= \sum_{i=0}^{M}\nu_{(i,j)}\pi_{(i,j)}\sum_{k=0}^{M}\left(\frac{q((i,j), (k,j+1))}{\nu_{(i,j)}}\right)\mathbb{E}_{(k,j+1)}\left[T_{(m,j+1)}^{A}\right] \\ &= \sum_{i=0}^{M}\pi_{(i,j)}R_{i,m},
\end{align*} thus proving that $\vec\pi_{j+1} = \vec\pi_{j}\textbf{R}$, since $R_{i,m}$ is $\nu_{(i,j)}$ times the expected amount of time the chain spends in state $(m,j+1)$ before returning to $L_{j}$, given it starts in state $(i,j)$, and $\textbf{R}$ is the rate matrix whose $(i,m)$th element is given by $R_{i,m}$.

We will soon see that the CAP method consists of applying Theorem \ref{thm:main} by choosing the set $A$ in a different manner, while simultaneously observing that the resulting expected values of the form $\rew zx$ can be reinterpreted as tractable expected values associated with an M/M/1/clearing model.


\subsection{Preliminaries}\label{sec:mainresult}

Our main results, and their proofs, will rely on the following notation:
\begin{itemize}
\item $P_m\equiv\{(m,j_0+1),(m,j_0+2),(m,j_0+3),\ldots\}$ is the set of states in phase $m$ with level $j\ge j_0+1$ (i.e., the set of states in phase $m$ of $\mathcal R$ excluding state $(m,j_0)$).
\item  $\rho_m\equiv\lambda_m/\mu_m$.
\item $\phi_m(\cdot)$ is the Laplace Transform of the busy period (time to first reach state 0, given that one starts in state 1) of an M/M/1 Markov chain with arrival rate $\lambda_m$ and departure rate $\mu_m$:
\[\phi_m(s)\equiv\frac{s+\lambda_m+\mu_m-\sqrt{(s+\lambda_m+\mu_m)^2-4\lambda_m\mu_m}}{2\lambda_m}.\]
\item  The \emph{bases} of our main theorem, $r_m$, are given by
\begin{align}
r_m&\equiv
\begin{cases}
\rho_m\phi_m(\alpha_m)&\mbox{if }\mu_m>0\\
\displaystyle{\frac{\lambda_m}{\lambda_m+\alpha_m}}&\mbox{if }\mu_m=0,\label{eq:rm}
\end{cases}
\end{align}
recalling that
\[\alpha_{m} \equiv \sum_{i={m+1}}^{M}\sum_{\Delta=-1}^1 \alp m{i-m}\Delta.\]
\item For convenience, we define the following quantity, which will appear frequently in our analysis:
\begin{align}\label{eq:omega}
\Omega_m&\equiv \frac{r_{m}}{\lambda_{m}(1 - r_{m}\phi_{m}(\alpha_{m}))}.
\end{align}
\end{itemize}
As a consequence of the ergodicity assumption on class $\mathbb M$ Markov chains, we have
\begin{itemize}
\item for any phase $m$, $\lambda_m\ge \mu_m$ implies $\alpha_{m} > 0$,
\item and for any phase $m$, $\lambda_m=0$ implies that there exists a phase, $i<m$, and $\Delta\in\{-1,0,1\}$ such that $\alp i{m-i}\Delta>0$.
\end{itemize}
We also make the following observations:
\begin{itemize}
\item $r_m<1$ for all phases, $m\in\{0,1,\ldots,M\}$.
\item $r_m=0$ if and only if $\lambda_m=0$.
\item $r_m=\rho_m$ whenever $\alpha_m=0$ (e.g., when $m=M$, as $\alpha_M=0$).  This is because $\alpha_m=0$ implies that $\mu_m>\lambda_m$ by the ergodicity assumption, which yields $\phi_m(0)=1$.
\item $\phi_m(s)=0$ for all $s$ whenever $\lambda_m>\mu_m=0$, which follows from the expression for $\phi_m(s)$.  Alternatively, this follows by observing that the busy period of a degenerate (non-ergodic) M/M/1 Markov chain with arrival rate $\lambda_m=0$ is infinite.
\end{itemize}

We have the following fundamental result on class $\mathbb M$ Markov chains, which together with Theorem \ref{thm:main}, will enable us to prove the main results of our paper (Theorems \ref{THM:RESULT}, \ref{THM:RESULT2}, and \ref{THM:RESULT3}).

\begin{theorem}\label{thm:eij}
For any Markov class $\mathbb M$ Markov chain, if $\lambda_m,\mu_m>0$ and $\ell,j\ge j_0+1$, we have
\begin{align}
\label{eq:thm}
\rep \ell jm&=\begin{cases}
\Omega_m r_m^{j-\ell}\left(1-(r_m\phi_m(\alpha_m))^{\ell-j_0}\right)&\mbox{if }\ell\le j\\
\Omega_{m}\phi_{m}(\alpha_{m})^{\ell - j}\left(1 - (r_{m}\phi_{m}(\alpha_{m}))^{j - j_{0}}\right)&\mbox{if }\ell\ge j.
\end{cases}
\end{align}
with $r_m$ as given in \eqref{eq:rm} and $\Omega_m$ as given in \eqref{eq:omega}.
\end{theorem}
\begin{proof}
The proof of this result is deferred to Section \ref{sec:clearing}, which is entirely focused on proving this result via clearing model analysis.
\end{proof}

The remainder of this section will present our main results, giving the stationary distribution of class $\mathbb M$ chains via the CAP method in three different cases.  In Section \ref{sec:result1}, we consider the case where all bases, $r_m$, are distinct whenever they are nonzero.  Distinct bases arise in many models where there is no structure connecting the transition rates associated with each phase.  For example, the class $\mathbb M$ Markov chain representing the ``server in different power states'' model presented in Section \ref{sec:sleep} has distinct bases.  In Section \ref{sec:result2} we consider the case where all bases are the same (i.e., $r_0=r_1=\cdots=r_m$), while requiring that $\lambda_m,\mu_m>0$, for simplicity.  We study this setting because it is the simplest case featuring repeated nonzero bases.  Finally, in Section \ref{sec:result3} we proceed to the case where all bases except for $r_M$ are the same (i.e., $r_0=r_1=\cdots=r_{M-1}\neq r_{M}$).  This structure, which is studied in \cite{johan06,johan09}, is common in settings where phase transitions are analogous across all phases, except for the final phase where there are no transitions to a further phase before the process transitions to the non-repeating portion.  In this case, we again assume that $\lambda_m,\mu_m>0$, for simplicity. While in principle, the CAP method can be used to determine the limiting probabilities of \emph{any} class $\mathbb M$ Markov chain, for simplicity, we do not cover other cases (e.g., $r_1=r_2\neq r_3=r_5=r_7\neq r_4=r_6\neq r_1$), as the computations become increasingly cumbersome.

\subsection{The case where all nonzero bases are distinct}\label{sec:result1}
We are now ready to present our main result for the case where all nonzero bases, $r_m$, are distinct.  Theorem \ref{THM:RESULT}  expresses the stationary distribution of such class $\mathbb M$ Markov chains as the solution to a finite system of linear equations.
\begin{theorem}\label{THM:RESULT}
For any class $\mathbb M$ Markov chain such that all nonzero bases $r_1,r_2,\ldots,r_M$---given in Equation \eqref{eq:rm}---are distinct (i.e., $r_m\neq r_i$ implies either $m\neq i$ or $r_m=\lambda_m=0$), for all $j\ge j_0+1$, we have a limiting probability distribution of the form \[\pi_{(m,j)}=\sum_{k=0}^m c_{m,k}r_k^{j-j_0},\] where $\{c_{m,k}\}_{0\le k\le m\le M}$ are constants with respect to $j$.  Moreover, together with $\{\pi_{(m,j_0)}\}_{0\le m\le M}$ and $\{\pi_x\}_{x\in\mathcal N}$, the $\{c_{m,k}\}_{0\le k\le m\le M}$ values constitute $M(M+5)/2+|\mathcal N|+2$ ``unknown variables'' satisfying the following system of $M(M+5)/2+|\mathcal N|+3$ linear equations:
{\scriptsize
\begin{empheq}[left=\empheqlbrace]{align*}
 c_{m,k}&=\frac{r_kr_m\left(\displaystyle{\sum_{i=k}^{m-1}\sum_{\Delta=-1}^1c_{i,k}\alp i{m-i}\Delta r_k^\Delta}\right)}{\lambda_m(r_k-r_m)(1-\phi_m(\alpha_m)r_k)}\hspace{2.45cm} (0\le k < m\le M\colon r_m,r_k>0)\\
c_{m,k}&=\frac{\displaystyle{\sum_{i=k}^{m-1}\sum_{\Delta=-1}^1c_{i,k}\alp i{m-i}\Delta r_k^\Delta}}{\mu_m(1-r_k)+\alpha_m}\hspace{3.44cm}(0\le k< m\le M\colon r_k>r_m=0)\\
c_{m,k}&=0 \hspace{8.04cm}(0\le k<m\le M\colon r_k=0)\\
c_{m,m}&=\pi_{(m,j_0)}-\sum_{k=0}^{m-1}c_{m,k} \hspace{7.28cm}(0\le m\le M)\\
\qquad\ \ \pi_{(m,j_0)}&=\frac{\displaystyle{\mu_m\sum_{k=0}^m c_{m,k}r_k+\sum_{x\in\mathcal N}q(x,(m,j_0))\pi_x}+\sum_{i=0}^{m-1}\sum_{\Delta=-1}^0\alp i{m-i}\Delta\pi_{(i,j_0-\Delta)}}{\displaystyle{\lambda_m+\sum_{i=m+1}^M\sum_{\Delta=0}^1\alp m{i-m}\Delta+\sum_{x\in\mathcal N}q((m,j_0),x)}}\hspace{0.3cm}(0\le m\le M)\\
\pi_x&=\frac{\displaystyle{\sum_{m=0}^M q((m,j_0),x)\pi_{(m,j_0)}+\sum_{y\in\mathcal N}q(y,x)\pi_y}}{\displaystyle{\sum_{m=0}^M q(x,(m,j_0))+\sum_{y\in\mathcal N}q(x,y)}}\hspace{5.16cm}(x\in\mathcal N)\\
1&=\sum_{x\in\mathcal N}\pi_x+\sum_{m=0}^{M}\sum_{k=0}^m\frac{c_{m,k}}{1-r_k},
\end{empheq}
}
where $q(x,y)$ denotes the transition rate from state $x$ to state $y$.
\end{theorem} We note before proving Theorem \ref{THM:RESULT} that solving this system of equations \emph{symbolically} will yield closed-form solutions for the limiting probabilities.  Alternatively, if all parameter values are fixed and known, an exact numerical solution can be found by solving the system numerically using exact methods.  Note that there is one more equation than there are unknowns, as is often the case in representations of limiting equations through balance equations.  Although one equation can be omitted from the system, the normalization equation must be used in order to guarantee a unique solution.

It is also worth observing that once the values $\{\pi_{x}\}_{x \in \mathcal{N}}$ and $\{\pi_{(m,j_{0})}\}_{0 \leq m \leq M}$ are known, all other $c_{m,k}$ terms can be computed recursively, without having to apply Gaussian elimination to the entire linear system given in Theorem \ref{THM:RESULT}.

This recursion may also simplify further for some types of class $\mathbb{M}$ Markov chains.  For example, if $\alp{m}{\Delta_1}{\Delta_2} = 0$ for all $\Delta_1 \geq 2$, $\Delta_2\in\{-1,0,1\}$, and $0 \leq m \leq M$, then when all bases are positive, for any $k<m$, we have
\begin{align*} c_{m,k} = c_{m-1,k}\frac{r_{k}r_{m}}{\lambda_{m}(r_{k} - r_{m})(1 - \phi_{m}(\alpha_{m})r_{k})}\sum_{\Delta = -1}^{1}\alp{m-1}1{\Delta} r_{k}^{\Delta} \end{align*} which further implies, for $k < m$,
\begin{align*} c_{m,k} = c_{k,k}\prod_{\ell = 1}^{m-k}\frac{r_{k}r_{k + \ell}}{\lambda_{k + \ell}(r_{k} - r_{k + \ell})(1 - \phi_{k + \ell}(\alpha_{k + \ell})r_{k})}\sum_{\Delta = -1}^{1}\alp{k + \ell - 1}1{\Delta} r_{k}^{\Delta} \end{align*} meaning that only the $\{c_{k,k}\}_{0\le k\le M}$ terms need to be computed recursively.


\begin{proof}[Proof of Theorem \ref{THM:RESULT}]


For simplicity, we present the proof for the case where $\lambda_m,\mu_m>0$ for all phases $m\in\{0,1,2,\ldots,M\}$.  The complete proof that includes the cases where one or both of $\lambda_m$ and $\mu_m$ may be 0 for some phases, $m$, is given in Appendix \ref{app:complete}.

We prove the theorem via strong induction on the phase, $m$.  Specifically, for each phase $m$, we will show that $\pi_{(m,j)}$ takes the form $\pi_{(m,j)}=\sum_{k=0}^m c_{m,k}r_k^{j-j_0}$ for all $j\ge j_0+1$, and show that $\{c_{m,k}\}_{0\le k\le m-1}$ satisfies
\[c_{m,k}=\displaystyle{\frac{r_{k}r_m}{\lambda_m(r_k-r_m)(1-\phi_m(\alpha_m)r_k)}\left(\sum_{i=k}^{m-1}\sum_{\Delta=-1}^1c_{i,k}\alp i{m-i}\Delta r_k^\Delta\right)}
\]
while $c_{m,m}=\pi_0-\sum_{k=0}^{m-1} c_{m,k}$.
Finally, after completing the inductive proof, we justify that the remaining linear equations in the proposed system are ordinary balance equations together with the normalization constraint.

\smallskip

\noindent\textbf{Base case:}

We begin our strong induction by verifying that the claim holds for the base case (i.e., for $m=0$).  
In this case, Equation \eqref{eq:thm} yields \[\rep {j_0+1}j0=\Omega_0 r_0^{j-j_0-1}(1-r_0\phi_0(\alpha_0))=\frac{r_0^{j-j_0}}{\lambda_0}.\]
We can now apply Theorem \ref{thm:main}, yielding
\begin{align*}
\pi_{(0,j)}&=\pi_{(0,j_0)}\lambda_0\rep {j_0+1}j0=\pi_{(0,j_0)}\lambda_0\left(\frac{r_0^{j-j_0}}{\lambda_0}\right)=\pi_{(0,j_0)}r_0^{j-j_0}\\
&=c_{0,0}r_0^{j-j_0},
\end{align*} where $c_{0,0}=\pi_{(0,j_0)}$.  Hence, $\pi_{(0,j)}$ takes the claimed form. Moreover, $c_{0,0}$ satisfies the claimed constraint as $c_{0,0}=\pi_{(0,j_0)}-\sum_{k=0}^{m-1} c_{m,k}=\pi_{(0,j_0)}-0=\pi_{(0,j_0)}$, because the sum is empty when $m=0$.  Note that when $m=0$, $\{c_{m,k}\}_{0\le k<m\le M}$ is empty, and hence, there are no constraints on these values that require verification.


\smallskip

\noindent\textbf{Helpful computations:}

Before proceeding to the inductive step, we compute two useful expressions: First, we have $\lambda_m \rep{j_0+1}jm=r_m^{j-j_0},$ which follows from applying Equation \eqref{eq:thm}.  Next, we have
{\small 
\begin{align*}
&\sum_{\ell=1}^\infty r_k^{\ell-j_0}\rep\ell jm= \sum_{\ell=j_0+1}^j r_k^{\ell-j_0}\rep \ell j m+\sum_{\ell=j+1}^\infty r_k^{\ell-j_0}\rep\ell jm\\
&\qquad=\Omega_m\left(\sum_{\ell=j_0+1}^j r_k^{\ell-j_0}r_m^{j-\ell}\left(1-(r_m\phi_m(\alpha_m))^{\ell-j_0}\right)\right.\\
&\qquad\qquad\qquad\left.+\sum_{\ell=j+1}^\infty r_k^{\ell-j_0}\phi_{m}(\alpha_m)^{\ell-j}\left(1-(r_m\phi_m(\alpha_m))^{j-j_0}\right)\right)\\
&\qquad=\frac{r_{k}r_m(r_k^{j-j_0}-r_m^{j-j_0})}{\lambda_m(r_k-r_m)(1-\phi_m(\alpha_m)r_k)},
\end{align*}
}
where the last equality follows from well known geometric sum identities.  Note that this expression is well-defined because $r_k\neq r_m$ by assumption and $r_m\phi_m(\alpha_m)\neq 1$.

\smallskip

\noindent\textbf{Inductive step:}

Next, we proceed to the inductive step and assume the induction hypothesis holds for all phases $i\in\{0,1,\ldots,m-1\}$.  In particular, we assume that $\pi_{(i,j)}=\sum_{k=0}^i c_{i,k} r_k^{j-j_0}$ for all $i<m$.
Applying Theorem \ref{thm:main}, the induction hypothesis, and our computations above, we have\footnote{
Note that we have also used the fact that $\pi_{(i,j_0)}$ also satisfies the claimed form for all $i<m$, which is true as $c_{i,i}=\pi_{(i,j_0)}-\sum_{k=0}^{i-1} c_{i,k}$ (from the inductive hypothesis) implies that $\pi_{(i,j_0)}=\sum_{k=0}^{i} c_{i,k}=\sum_{k=0}^{i} c_{i,k}r_k^0$.}
{\footnotesize
\begin{align*}
\pi_{(m,j)}&=\pi_{(m,j_0)}\lambda_m\rep{j_0+1}jm+\sum_{i=0}^{m-1}\sum_{\ell=1}^\infty\sum_{\Delta=-1}^1\pi_{(i,\ell-\Delta)}\alp i{m-i}\Delta\rep {\ell} j m\nonumber\\
&=\pi_{(m,j_0)}r_m^{j-j_0}+\sum_{i=0}^{m-1}\sum_{\ell=1}^\infty\sum_{\Delta=-1}^1\alp i{m-i}\Delta\left(\sum_{k=0}^{i}c_{i,k}r_k^{\ell-j_0-\Delta}\rep \ell j m\right)\nonumber\\
&=\pi_{(m,j_0)}r_m^{j-j_0}+\sum_{k=0}^{m-1}\sum_{i=k}^{m-1}\left(c_{i,k}\sum_{\Delta=-1}^1\alp i{m-i}\Delta r_k^\Delta\right)\left(\sum_{\ell=1}^\infty r_k^{\ell-j_0}\rep \ell j m\right)\nonumber\\
&=\pi_{(m,j_0)}r_m^{j-j_0}+\sum_{k=0}^{m-1}\sum_{i=k}^{m-1}\left(c_{i,k}\sum_{\Delta=-1}^1\alp i{m-i}\Delta r_k^\Delta\right)\left(\frac{r_k r_m(r_k^{j-j_0}-r_m^{j-j_0})}{\lambda_m(r_k-r_m)(1-\phi_m(\alpha_m)r_k)}\right)\\
&=\sum_{k=0}^m c_{m,k}r_k^{j-j_0},
\end{align*}
}
where we have collected terms with
{\small
\begin{align*}
c_{m,k}&=\frac{\displaystyle{r_kr_m\left(\sum_{i=k}^{m-1}\sum_{\Delta=-1}^1c_{i,k}\alp i{m-i}\Delta r_k^\Delta\right)}}{\lambda_m(r_k-r_m)(1-\phi_m(\alpha_m)r_k)}& (&0\le k < m\le M)
\end{align*}
}
and $c_{m,m}=\pi_{(m,j_0)}-\sum_{k=0}^{m-1}c_{m,k}$, as claimed.  This completes the inductive step and the proof by induction.


\smallskip

\noindent\textbf{The balance equations and normalization constraint:}

The equations with $\pi_{(m,j_0)}$ and $\pi_x$ in their left-hand sides in our proposed system are ordinary balance equations (that have been normalized so that there are no coefficients on the left-hand side).

It remains to verify that the final equation, which is the normalization constraint:
\begin{align*}
1&=\sum_{x\in\mathcal N}\pi_x+\sum_{m=0}^M\pi_{(m,j_0)}+\sum_{m=0}^M\sum_{j=j_0+1}^\infty \pi_{(m,j)}\\
&=\sum_{x\in\mathcal N}\pi_x+\sum_{m=0}^M\sum_{k=0}^M c_{m,k}+\sum_{m=0}^M\sum_{k=0}^{m-1}\sum_{j=j_0+1}^\infty c_{m,k}r_k^{j-j_0}\\
&=\sum_{x\in\mathcal N}\pi_x+\sum_{m=0}^M\sum_{k=0}^{m} \frac{c_{m,k} r_k}{1-r_k}.
\end{align*}
\end{proof}

\subsection{The case where all bases agree}\label{sec:result2}

The CAP method can also be used in cases where some of the base terms coincide.  We assume, for the sake of readability, that $\lambda_{m}$ and $\mu_{m}$ are both positive for each phase $m$, but analogous results can still be derived when this is no longer the case.

In order to derive our result, we will make use of the following lemma: we omit the proof, but each formula can be derived using the lemmas contained in Appendix \ref{app:negbinlems}.


\begin{lemma} \label{lem:seriesequal} For a class $\mathbb M$ Markov chain with all $\lambda_m,\mu_m>0$ and $r_{0} = r_{1}=\cdots=r_{M}$, for each integer $u \geq 0$ and each integer $j \geq j_{0} + 1$, we have the following three identities:
\begin{align*}
\mbox{\textbullet}&\sum_{\ell = j_{0} + 2}^{\infty}{\ell - (j_{0} + 1) + u \choose u}r_{0}^{\ell - j_{0}}\rep{\ell-1}jm\\
&\quad= \sum_{k=1}^{u+1}\frac{\Omega_{m}r_{0}}{(1 - r_{0}\phi_{m}(\alpha_{m}))^{u + 1 - k}}{j - (j_{0} + 1) + k\choose k}r_{0}^{j - j_{0}},\qquad\qquad\qquad\qquad\qquad
\end{align*}
\begin{align*}
\mbox{\textbullet}& \sum_{\ell = j_{0} + 1}^{\infty}{\ell - (j_{0} + 1) + u \choose u}r_{0}^{\ell - j_{0}}\rep\ell jm \\
&\quad= \Omega_{m}{j - (j_{0} + 1) + u + 1 \choose u+1}r_{0}^{j - j_{0}}\\
&\qquad+ \sum_{k=1}^{u}\frac{\Omega_{m}r_{0}\phi_{m}(\alpha_{m})}{(1 - r_{0}\phi_{m}(\alpha_{m}))^{u + 1 - k}}{j - (j_{0} + 1) + k \choose k}r_{0}^{j - j_{0}},\qquad\qquad\qquad\qquad\qquad
\end{align*}
\begin{align*}
\mbox{\textbullet}&\sum_{\ell = j_{0} + 1}^{\infty}{\ell - (j_{0} + 1) + u \choose u}r_{0}^{\ell - j_{0}}\rep{\ell+1}jm \\
&\quad= \frac{\Omega_{m}}{r_{0}}{j - (j_{0} + 1) + u + 1 \choose u+1}r_0^{j-j_0}- {j - (j_{0} + 1) + u \choose u}\frac{r_{0}^{j - j_{0}}}{\lambda_m}\qquad\qquad\qquad\qquad\qquad\\
&\qquad+ \sum_{k = 1}^{u}\frac{\Omega_{m}r_{0}\phi_{m}(\alpha_{m})^{2}}{(1 - r_{0}\phi_{m}(\alpha_{m}))^{u + 1 - k}}{j - (j_{0} + 1) + k \choose k}r_{0}^{j - j_{0}}.
\end{align*}
\end{lemma}

\begin{theorem}\label{THM:RESULT2} For a class $\mathbb M$ Markov chain with all $\lambda_m,\mu_m>0$ and $r_0=r_1=\cdots=r_M$, for all $0 \leq m \leq M$, $j \geq j_{0}$, we have
\[ \pi_{(m,j)} = \sum_{k=0}^{m}c_{m,k}{j - (j_{0} + 1) + k \choose k}r_{0}^{j - j_{0}}\] where the $\{c_{m,k}\}_{0 \le k \le m \le M}$ values satisfy the system of linear equations
{\small
\begin{empheq}[left=\empheqlbrace]{align*}
\ c_{m,0} &= \pi_{(m,j_{0})},&(& 0 \le m \le M)\\
 \ c_{m,k} &= \Omega_{m}r_{0}\sum_{u=k}^{m-1}\sum_{i=u}^{m-1}c_{i,u}\left[\sum_{\Delta = -1}^{1}\frac{\displaystyle{\alp{i}{m - i}{\Delta}\phi_{m}(\alpha_{m})^{\Delta + 1}}}{(1 - r_{0}\phi_{m}(\alpha_{m}))^{u + 1 - k}}\right]&& \\
 &\qquad - \frac{1}{\lambda_{m}}\sum_{i=k}^{m-1}c_{i,k}\alp{i}{m-i}{1} &&\\
 &\qquad +\Omega_{m}\sum_{i=k-1}^{m-1}c_{i,k-1}\left[\sum_{\Delta = -1}^{1}\alp{i}{m - i}{\Delta}r_{0}^{-\Delta}\right]&(&1\le k\le m-1)\\
 \ \ c_{m,m} &= c_{m-1, m-1}\Omega_{m}\sum_{\Delta = -1}^{1}\alp {m-1}{1}{\Delta}r_{0}^{\Delta}&(&1 \le m \le M),
\end{empheq}}
together with the usual balance equations and normalization constraint.
\end{theorem}

\begin{proof} Starting with phase $0$, we observe as before that, for $j \geq j_{0} + 1$,
\begin{align*} \pi_{(0,j)} = \pi_{(0,j_{0})}\lambda_{0}\rep {j_0+1}j0 = \pi_{(0,j_{0})}r_{0}^{j - j_{0}} \end{align*}
and this equality is clearly also valid when $j = j_{0}$.

We now proceed by induction.  Assuming the result holds for $\pi_{(i,\ell)}$ for $0 \leq i \leq m-1$, $\ell \geq j_{0}$, we have
{\small \begin{align*}
\pi_{(m,j)} &= \pi_{(m,j_{0})}\lambda_{m}\rep{j_0+1}jm + \sum_{i=0}^{m-1}\sum_{\ell = j_{0} + 2}^{\infty}\pi_{(i,\ell)}\alp{i}{m-i}{-1}\rep{\ell-1}jm \\
&\qquad+ \sum_{i=0}^{m-1}\sum_{\ell = j_{0} + 1}^{\infty}\pi_{(i,\ell)}\alp{i}{m-i}0\mathbb{E}_{(m,\ell)}[T_{(m,j)}^{P_{m}}] \\
&\qquad+ \sum_{i=0}^{m-1}\sum_{\ell = j_{0}}^{\infty} \pi_{(i,\ell)}\alp{i}{m-i}1\mathbb{E}_{(m,\ell+1)}[T_{(m,j)}^{P_{m}}] \\
&= \pi_{(m,j_{0})}r_{0}^{j - j_{0}} + \sum_{i=0}^{m-1}c_{i,0}\Omega_{m}\left[\sum_{\Delta = -1}^{1}\alp{i}{m-i}{-1}r_{0}^{-\Delta}\right]{j - (j_{0} + 1) + 1 \choose 1}r_{0}^{j - j_{0}} \\
&\qquad+ \sum_{u=1}^{m-1}\sum_{i=u}^{m-1}c_{i,u}\alp i{m-i}{-1}\sum_{\ell = j_{0} + 2}^{\infty}{\ell - (j_{0} + 1) + u \choose u}r_{0}^{\ell - j_{0}}\rep{\ell-1}jm \\
&\qquad+ \sum_{u=1}^{m-1}\sum_{i=u}^{m-1}c_{i,u}\alp{i}{m-i}0\sum_{\ell = j_{0} + 1}^{\infty}{\ell - (j_{0} + 1) + u \choose u}r_{0}^{\ell - j_{0}}\rep\ell jm \\
&\qquad+ \sum_{u=1}^{m-1}\sum_{i=u}^{m-1}c_{i,u}\alp{i}{m-i}1\sum_{\ell = j_{0} + 1}^{\infty}{\ell - (j_{0} + 1) + u \choose u}r_{0}^{\ell - j_{0}}\rep{\ell+1}jm
\end{align*}}
and after applying Lemma \ref{lem:seriesequal} and simplifying, we conclude that
{\footnotesize \begin{align*} \pi_{(m,j)} &= \pi_{(m,j_{0})}r_{0}^{j - j_{0}} \\
&+ \sum_{k=0}^{m-1}\sum_{i=k}^{m-1}c_{i,k}\Omega_{m}\left[\sum_{\Delta = -1}^{1}\alp{i}{m-i}{\Delta}r_{0}^{-\Delta}\right]{j - (j_{0} + 1) + k + 1 \choose k + 1}r_{0}^{j - j_{0}} \\
&+ \sum_{k=1}^{m-1}\sum_{u=k}^{m-1}\sum_{i=u}^{m-1}c_{i,u}\Omega_{m}r_{0}\left[\sum_{\Delta = -1}^{1}\frac{\alp{i}{m-i}{\Delta}\phi_{m}(\alpha_{m})^{-\Delta + 1}}{(1 - r_{0}\phi_{m}(\alpha_{m}))^{u + 1 - k}}\right]{j - (j_{0} + 1) + k \choose k}r_{0}^{j - j_{0}} \\ &- \frac{1}{\lambda_{m}}\sum_{k=1}^{m-1}\sum_{i=k}^{m-1}c_{i,k}\alp{i}{m-i}{1}{j - (j_{0} + 1) + k \choose k}r_{0}^{j - j_{0}} \\
&= \sum_{k=0}^{m}c_{m,k}{j - (j_{0} + 1) + k \choose k}r_{0}^{j - j_{0}},
\end{align*}}
where we have collected terms so that for $1\le k\le m-1$ we have
\begin{align*}
c_{m,k} &= \sum_{u=k}^{m-1}\sum_{i=u}^{m-1}c_{i,u}\Omega_{m}r_{0}\left[\sum_{\Delta = -1}^{1}\frac{\alp{i}{m - i}{\Delta}\phi_{m}(\alpha_{m})^{-\Delta + 1}}{(1 - r_{0}\phi_{m}(\alpha_{m}))^{u + 1 - k}}\right] \\
&- \frac{1}{\lambda_{m}}\sum_{i=k}^{m-1}c_{i,k}\alp{i}{m-i}1 \\
&+ \Omega_{m}\sum_{i=k-1}^{m-1}c_{i,k-1}\left[\sum_{\Delta = -1}^{1}\alp{i}{m-i}{\Delta}r_{0}^{-\Delta}\right],
\end{align*}
while $c_{m,0}=\pi_{(m,j_0)}$ and $c_{m,m} = c_{m-1, m-1}\Omega_{m}\left[\sum_{\Delta = -1}^{1}\alp{m-1}{1}{\Delta}r_{0}^{-\Delta}\right]$, as claimed.
\end{proof}

\subsection{The case where all bases except $r_M$ agree}\label{sec:result3}

We conclude this section by considering the case where $r_{0} = r_{1} = \cdots = r_{M-1} \neq r_{M}$, as this case is satisfied by the Markov chains studied in \cite{johan06, johan09}.  The following lemma can be used to compute the limiting probability distribution.   The proof is again omitted, but as with Lemma \ref{lem:seriesequal}, each formula can be derived using the lemmas contained in Appendix \ref{app:negbinlems}.

\begin{lemma} For a class $\mathbb M$ Markov chain with all $\lambda_m,\mu_m>0$ and $r_{0}=r_{1}=\cdots=r_{M-1}\neq r_{M}$, for each integer $u \geq 0$, we have the following three identities:
{\small
\begin{align*}
\mbox{\textbullet}&\sum_{\ell = j_{0} + 2}^{\infty}{\ell - (j_{0} + 1) + u \choose u}r_{0}^{\ell - j_{0}}\rep{\ell-1}jM \\
&\quad= -\Omega_{M}r_{0}\left[\frac{1}{(1 - r_{0})^{u+1}}-\frac{1}{(1 - \frac{r_{0}}{r_{M}})^{u+1}} \right]r_{M}^{j - j_{0}} \\
&\qquad+ \sum_{k = 0}^{u}\Omega_{M}r_{0}\left[\frac{1}{(1 - r_{0})^{u+1-k}}-\frac{1}{(1 - \frac{r_{0}}{r_{M}})^{u+1-k}}\right]{j - (j_{0} + 1) + k \choose k}r_{0}^{j - j_{0}},
\end{align*}
\begin{align*}
\mbox{\textbullet}&\sum_{\ell = j_{0}+1}^{\infty}{\ell - (j_{0} + 1) + u \choose u}r_{0}^{\ell - j_{0}}\rep\ell jM \\
&\quad= -\Omega_{M}r_{0}\left[\frac{1}{(1 - r_{0})^{u+1}} - \frac{1}{r_{M}(1 - \frac{r_{0}}{r_{M}})^{u+1}}\right]r_{M}^{j - j_{0}} \\
&\qquad+ \sum_{k=0}^{u}\Omega_{M}r_{0}\left[\frac{1}{(1 - r_{0})^{u+1-k}} - \frac{1}{r_{M}(1 - \frac{r_{0}}{r_{M}})^{u+1-k}}\right]{j - (j_{0} + 1) + k \choose k}r_{0}^{j - j_{0}}, \\
\end{align*}
\begin{align*}
\mbox{\textbullet}&\sum_{\ell = j_{0} + 1}^{\infty}{\ell - (j_{0} + 1) + u \choose u}r_{0}^{\ell - j_{0}}\rep{\ell+1}jM \\
&\quad= -\Omega_{M}r_{0}\left[\frac{1}{(1 - r_{0})^{u+1}} - \frac{1}{r_{M}^{2}(1 - \frac{r_{0}}{r_{M}})^{u+1}}\right]r_{M}^{j - j_{0}}- {j - (j_{0} + 1) + u \choose u}\frac{r_{0}^{j - j_{0}}}{\lambda_M}\\
&\qquad+\sum_{k=0}^{u}\Omega_{M}r_{0}\left[\frac{1}{(1 - r_{0})^{u+1 - k}} - \frac{1}{r_{M}^{2}(1 - \frac{r_{0}}{r_{M}})^{u+1-k}}\right]{j - (j_{0} + 1) + k \choose k}r_{0}^{j - j_{0}}.
\end{align*}}
\end{lemma}

Our next theorem gives an expression for the stationary distribution of a class $\mathbb M$ Markov chain when $r_{0} = r_{1} = \cdots = r_{M-1} \neq r_{M}$.  As the proof is similar to those of Theorems \ref{THM:RESULT} and \ref{THM:RESULT2}, we omit the proof.

\begin{theorem}\label{THM:RESULT3} Suppose a class $\mathbb M$ Markov chain has all $\lambda_m,\mu_m>0$ and $r_{0}=r_{1}=\cdots=r_{M-1}\neq r_{M}$.  Then, for all $0 \leq m \leq M-1$, $j \geq j_{0}$,
\begin{align*}
\pi_{(m,j)} &= \sum_{k=0}^{m}c_{m,k}{j - (j_{0} + 1) + k \choose k}r_{0}^{j - j_{0}},\\
\pi_{(M,j)} &= \sum_{k=0}^{M-1}c_{M,k}{j - (j_{0} + 1) + k \choose k}r_{0}^{j - j_{0}} + c_{M,M}r_{M}^{j - j_{0}}
\end{align*}
where the $\{c_{m,k}\}_{0 \leq k\le m \le M}$ values satisfy the system of linear equations
{\footnotesize
\begin{empheq}[left=\empheqlbrace]{align*}
c_{m,0} &= \pi_{(m,j_{0})}&(&0\le m < M)\\
c_{m,k} &= \Omega_{m}r_{0}\sum_{u=k}^{m-1}\sum_{i=u}^{m-1}c_{i,u}\left[\sum_{\Delta=-1}^1\frac{\alp{i}{m - i}{\Delta}\phi_{m}(\alpha_{m})^{\Delta + 1}}{(1 - r_{0}\phi_{m}(\alpha_{m}))^{u + 1 - k}}\right]&&\\
&\qquad + \Omega_{m}\sum_{i=k-1}^{m-1}c_{i,k-1}\left[\sum_{\Delta = -1}^{1}\alp{i}{m - i}{\Delta})r_{0}^{-\Delta}\right]&&\\
&\qquad - \frac{1}{\lambda_{m}}\sum_{i=k}^{m-1}c_{i,k}\alp{i}{m-i}{1}&(&1\le k< m<M)\\
c_{m,m} &= c_{m-1, m-1}\Omega_{m}\left[\sum_{\Delta = -1}^{1}\alp{m-1}{1}{\Delta}r_{0}^{\Delta}\right]&(&1\le m< M)\\
c_{M,0} &= \sum_{i=0}^{M-1}c_{i,0}\Omega_{M}r_{0}\left[\sum_{\Delta = -1}^{1}\left[\frac{1}{1 - r_{0}} - \frac{1}{r_{M}^{\Delta + 1}(1 - \frac{r_{0}}{r_{M}})}\right]\alp{i}{m-i}{\Delta}\right] && \\
&\qquad+ \sum_{u=1}^{M-1}\sum_{i=u}^{M-1}c_{i,u}\Omega_{M}r_{0}\left[\sum_{\Delta = -1}^{1}\alp{i}{M-i}{\Delta}\right[\frac{1}{(1 - r_{0})^{u+1}}&&\\
&\qquad\qquad\left.\left.- \frac{1}{r_{M}^{\Delta + 1}(1 - \frac{r_{0}}{r_{M}})^{u+1}}\right]\right]&&\\
c_{M,k} &= -\sum_{i=k}^{M-1}c_{i,k}\frac{\alp{i}{M - i}{1}}{\lambda_{M}} \\
&\qquad+ \sum_{u=k}^{M-1}\sum_{i=u}^{M-1}c_{i,u}\Omega_{M}r_{0}\left[\sum_{\Delta = -1}^{1}\alp{i}{m-i}{\Delta}\right[\frac{1}{(1 - r_{0})^{u+1-k}}&&\\
&\qquad\qquad\left.\left.- \frac{1}{r_{M}^{\Delta + 1}(1 - \frac{r_{0}}{r_{M}})^{u+1-k}}\right]\right]&(&1\le k<M)\\
\quad\qquad\  c_{M,M}&= \pi_{(M,j_{0})} - c_{M,0},&&
\end{empheq}
}
together with the usual balance equations and normalization constraint.
\end{theorem}

\section{Analysis of the M/M/1/clearing model}\label{sec:clearing}

In this section we pre\-sent an analysis of the M/M/1/clearing model Markov chain in order to prove Theorem \ref{thm:eij} (presented in Section \ref{sec:mainresult}), which we used in the proof of Theorems \ref{THM:RESULT} \ref{THM:RESULT2}, and \ref{THM:RESULT3}.  This analysis provides the framework on which the CAP method is built.

Like the ordinary M/M/1 model, the M/M/1/clearing model Markov chain (see Fig. \ref{fig:clearing}) has state space $\{0,1,2,3,\ldots\}$, with an arrival rate of $\lambda\equiv q(j,j+1)$ (for all $j\ge0$) and a departure rate of $\mu\equiv q(j,j-1)$ (for all $j\ge2$).  In addition, all nonzero states in the M/M/1/clearing model have an additional transition to state 0 representing a \emph{clearing} (also known as a catastrophe or disaster).  All clearing transitions occur with the same rate $\alpha\equiv q(j,0)$ (for all $j\ge2$), which we call the \emph{clearing rate}.  Note that from state 1, there are two ``ways'' of transitioning to state 0---a departure or a clearing---and hence, $q(1,0)=\mu+\alpha$.  We observe that each phase, $m$, of a class $\mathbb M$ Markov chain (for levels $j\ge j_0+1$) behaves like an M/M/1/clearing Markov chain, with clearing rate \[\alpha_{m} \equiv \sum_{i={m+1}}^{M}\sum_{\Delta=-1}^1 \alp m{i-m}\Delta,\] except with ``clearings'' transitioning to a different phase.

\begin{figure}
\begin{center}
\input{clearing.tikz}
\end{center}
\caption{Markov chain for the M/M/1/clearing model.  For any state $j\ge0$, there is a clearing rate with rate $\alpha$.  Note that the transition rate from state 1 to state 0 is $\mu+\alpha$ as either a departure or a clearing can cause this transition.  The thicker arrow denotes a \emph{set} of transitions.}
\label{fig:clearing}
\end{figure}
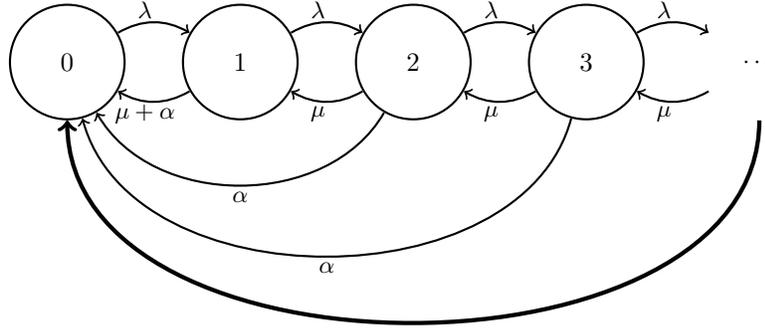

\subsection{Preliminary results on clearing models}

In this section we present  two preexisting results from the literature that will aid us in proving Theorem \ref{thm:eij}.  Our first result gives the limiting probability distribution of the M/M/1/clearing model:  see e.g., Corollary 4.2.2 of \cite{AbateWhitt}, as well as Exercise 10.7 of \cite{harchol2013performance}.

\begin{lemma}\label{lem:limprob}
In an M/M/1/clearing model with arrival, departure, and clearing rates $\lambda$, $\mu$, and $\alpha$, respectively, the limiting probability distribution is given by
\[\pi_j=(1-\rho\phi(\alpha))(\rho\phi(\alpha)^j),\] where $\rho=\lambda/\mu$ and $\phi(\cdot)$ is the Laplace transform of the busy period of an M/M/1 system: \[\phi(s)=\frac{s+\lambda+\mu-\sqrt{(s+\lambda+\mu)^2-4\lambda\mu}}{2\lambda}.\]
\end{lemma}
\begin{proof}
See the proof of Corollary 4.2.2 of \cite{AbateWhitt}.
\end{proof}

The next result is also known, and gives an expression for a probability that is useful in computing values of the form $\rew \ell j$ in the M/M/1 clearing model.  A similar result, presented in the context of Brownian motion, is given in Problems 22 and 23 from Chapter 7 of \cite{karlin1975first}.
\begin{lemma}\label{lem:pab}
In an M/M/1/clearing model with arrival, departure, and clearing rates $\lambda$, $\mu$, and $\alpha$, respectively, the probability that one reaches state $j>0$ before state $0$, given that one starts in state $\ell>0$, is given by
\[p_{\ell\to j}=\begin{cases}
\displaystyle{\frac{(\rho\phi(\alpha))^{j-\ell}(1-(\rho\phi(\alpha)^2)^\ell)}{1-(\rho\phi(\alpha)^2)^j}}&\mbox{if }\ell\le j\\
\phi(\alpha)^{\ell-j}&\mbox{if }\ell\ge j.
\end{cases}\]
\end{lemma}

\begin{proof}

First, note that in the degenerate case where $\ell=j$, we are already at state $j$ from the start, and so we reach state $j$ before reaching state $0$ \emph{surely}, yielding $p_{\ell\to j}=1$.  Substituting $\ell=j$ in either branch of the claimed expression for $p_{\ell\to j}$ yields 1, validating the claim in this case.

Next, we consider the case where $\ell>j$, which will be the simpler of the two remaining cases.  In this case, $p_{\ell\to j}$ can be viewed as the probability that the sum of $\ell-j$ independent M/M/1 busy periods (without clearing), $B_1,B_2,\ldots,B_{\ell-j}$, do not exceed the exponentially distributed ``clearing'' random variable $\zeta_\alpha$: \[p_{\ell\to j}=\mathbb P\left(\sum_{n=1}^{j-\ell}B_n\le \zeta_\alpha\right)=\mathbb{E}\left[e^{-\alpha \sum_{n=1}^{j - \ell}B_{n}}\right] = \phi(\alpha)^{j - \ell}\] 
as claimed, with the next-to-last equality following from the alternate interpretation of the Laplace Transform (see Appendix \ref{app:alternate} for details).   

Now let us consider the remaining case where $\ell<j$.  In this case, it will be helpful to consider two Poisson processes, one associated with arrivals, occurring with rate $\lambda$, and the other associated with departures, occurring with rate $\mu$.  Departures can happen even at state 0, although at state 0 departures \emph{do not cause a change of state}.  Let $N_A(t)$ and $N_D(t)$ be the number of such arrivals and departures during time interval $[0,t]$, assuming that we are in state $\ell$ at time 0.

Next, let $\tau_0=\inf\{t\colon \ell+(N_A(t)-N_D(t))=0\}$ be the first time after 0 until we have $\ell$ departures in excess of arrivals, and let $\tau_j=\inf\{t\colon \ell+(N_A(t)-N_D(t))=j\}$ be the first time after 0 until we have $j-\ell$ arrivals in excess of departures.  Although there may be positive probability that one of of these two events may never happen (i.e., $\max\{\tau_0,\tau_j\}=+\infty$), at least one of these events will happen almost surely.  Moreover, if either of these events happens before a clearing, which will occur at time $\zeta_\alpha\sim\ed(\alpha)$ (independent of both $\tau_0$ and $\tau_j$), then $\tau_0$ and $\tau_j$ describe the first time that we will reach state $0$ and $j$, respectively.

Given this notation, we can express $p_{\ell\to j}$, the probability that one next reaches state $j>\ell$ before state $0$ in an M/M/1/clearing model, given that one starts in state $\ell>0$, by
\begin{align*}
p_{\ell\to j}=\pr(\tau_j\le\min\{\tau_0,\zeta_\alpha\})=\ep[e^{-\alpha\tau_0}\cdot I\{\tau_0<\tau_j\}],
\end{align*}
where $I\{\cdot\}$ is the indicator function.  Similarly, if we let $p_{\ell\not\to j}$ be the probability that we reach 0---via departures, rather than via a clearing---before reaching $j$ and before a clearing, we have
\[p_{\ell\not\to j}=\pr(\tau_0\le\min\{\tau_j,\zeta_\alpha\})=\ep[e^{-\alpha\tau_j}\cdot I\{\tau_j<\tau_0\}].\]

At this point, it will be useful to compute the quantities $\ep[e^{-\alpha\tau_0}]$ and $\ep[e^{-\alpha\tau_j}]$.  Observe that $\tau_0$ is the time until we first have $\ell$ departures in excess of arrivals.  We can think of each time ``departures minus arrivals'' increments by one as the completion of an M/M/1 busy period.  Hence, $\tau_0$ corresponds to the time until we have completed $\ell$ consecutive independent busy periods.  Meanwhile, $\tau_j$ is the time until we first have $j-\ell>0$ arrivals in excess of departures.  Just as we can think each time ``departures minus arrivals'' increments by one as the completion of an M/M/1 busy period, we can also think of the each time ``arrivals minus departures'' increments by one as the completion of an M/M/1 busy period where we think of arrivals as ``departures'' occurring with rate $\lambda$ and departures as ``arrivals'' occurring with rate $\mu$.  Hence, $\tau_j$ corresponds to the time until we have completed $j-\ell$ consecutive independent busy periods with arrival rate $\mu$ and departure rate $\lambda$.  Consequently \[\ep[e^{-\alpha\tau_0}]=\phi(\alpha)^\ell,\qquad\ep[e^{-\alpha\tau_j}]=\eta(\alpha)^{j-\ell},\] where $\phi(\cdot)$ and $\eta(\cdot)$ are the Laplace transforms of the M/M/1 busy periods with arrival and departure rate pairs $(\lambda,\mu)$ and $(\mu,\lambda)$, respectively.  We observe that for all $s>0$,
\begin{align*}
\eta(s)&=\frac{s+\mu+\lambda-\sqrt{(s+\mu+\lambda)^2-4\mu\lambda}}{2\mu}\\
&=\left(\frac{\lambda}{\mu}\right)\left(\frac{s+\lambda+\mu-\sqrt{(s+\lambda+\mu)^2-4\lambda\mu}}{2\lambda}\right)\\
&=\rho\phi(s).
\end{align*}
Note that in the case that $\rho\neq1$, we must have $\eta(0)\neq\phi(0)$,  and in particular one of these transforms will not evaluate to 1.  This is not a problem as if $\rho<1$ (respectively, $\rho>1$), the underlying random variable of $\eta$ (respectively, $\phi$) has positive probability mass at infinity, and will thus not satisfy the ``usual'' condition of Laplace transforms evaluating to 1 at 0.

We proceed to use these expectations to determine $p_{\ell\to j}$:
\begin{align*}
\phi(\alpha)^\ell&=\ep[e^{-\alpha\tau_0}]\\
&=\ep[e^{-\alpha\tau_0}\cdot I\{\tau_0<\tau_j\}]+\ep[e^{-\alpha\tau_0}\cdot I\{\tau_j<\tau_0\}]\\
&=\ep[e^{-\alpha\tau_0}\cdot I\{\tau_0<\tau_j\}]+\phi(\alpha)^j\cdot\ep[e^{-\alpha\tau_j}\cdot I\{\tau_j<\tau_0\}]\\
&=p_{\ell\not\to j}+\phi(\alpha)^j p_{\ell\to j},\\
(\rho\phi(\alpha))^{j-\ell}&=\ep[e^{-\alpha\tau_j}]\\
&=\ep[e^{-\alpha\tau_j}\cdot I\{\tau_0<\tau_j\}]+\ep[e^{-\alpha\tau_j}\cdot I\{\tau_j<\tau_0\}]\\
&=(\rho\phi(\alpha))^j\cdot\ep[e^{-\alpha\tau_0}\cdot I\{\tau_0<\tau_j\}]+\ep[e^{-\alpha\tau_j}\cdot I\{\tau_j<\tau_0\}]\\
&=(\rho\phi(\alpha))^j(p_{\ell\not\to j})+p_{\ell\to j}.
\end{align*}
We justify $\ep[e^{-\alpha\tau_0}\cdot I\{\tau_j<\tau_0\}]=\phi(\alpha)^j\cdot \ep[e^{-\alpha\tau_j}\cdot I\{\tau_j<\tau_0\}]$ by observing that given that $\tau_j<\tau_0$, we reach state $j$ before state $0$ (ignoring clearings), so we can only reach state 0 by performing $j$ consecutive busy periods after reaching $j$.  We justify the analogous equality $\ep[e^{-\alpha\tau_j}\cdot I\{\tau_0<\tau_j\}]=(\rho\phi(\alpha))^j\cdot\ep[e^{-\alpha\tau_0}\cdot I\{\tau_0<\tau_j\}]$ by observing that given that $\tau_0<\tau_j$, we reach state $0$ before state $j$ (ignoring clearings), so we can only reach state $j$ by performing $j$ consecutive ``busy'' periods in an M/M/1 model with arrival rate $\mu$ and departure rate $\lambda$.

We now have a system of two linear equations in the two unknowns, $p_{\ell\to j}$ and $p_{\ell\not\to j}$.  Solving the system for $p_{\ell\to j}$ and simplifying, we find that \[p_{\ell\to j}=\frac{(\rho\phi(\alpha))^{j-\ell}-(\rho\phi(\alpha))^j\phi(\alpha)^\ell}{1-(\rho\phi(\alpha))^j\phi(\alpha)^j}=\frac{(\rho\phi(\alpha))^{j-\ell}(1-(\rho\phi(\alpha)^2)^\ell)}{1-(\rho\phi(\alpha)^2)^j},\] which proves the claim.
\end{proof}

\subsection{Applying clearing model analysis toward proving Theorem 2}

We now use Lemmas \ref{lem:limprob} and \ref{lem:pab} to compute $\rew \ell j$ in an M/M/1/clearing model, where $A$ is the set of nonzero states.  This result is presented in Lemma \ref{lem:eij}.  Finally, we will recast Lemma \ref{lem:eij} in the context of class $\mathbb M$ Markov chains, allowing us to prove Theorem \ref{thm:eij} from Section \ref{sec:mainresult}.

\begin{lemma}\label{lem:eij}
In an M/M/1/clearing model with arrival, departure, and clearing rates $\lambda$, $\mu$, and $\alpha$, respectively, if $A=\{1,2,3,\ldots\}$ denotes the set of nonzero states of the state space of the underlying Markov chain, then
\[\rew \ell j=\begin{cases}
\displaystyle{\frac{(\rho\phi(\alpha))^{j-\ell+1}\left(1-(\rho\phi(\alpha)^2)^\ell\right)}{\lambda(1-\rho\phi(\alpha)^2)}}&\mbox{if }\ell\le j\\ \\
\displaystyle{\frac{\rho\phi(\alpha)^{\ell-j+1}\left(1-(\rho\phi(\alpha)^2)^j\right)}{\lambda(1-\rho\phi(\alpha)^2)}}&\mbox{if }\ell\ge j.
\end{cases}\]
\end{lemma}
\begin{proof}
We first consider the case where $\ell\le j$.  We claim that
\begin{align}\label{eq:probclear}
\rew 1j=(p_{1\to \ell})\rew \ell j,
\end{align}
recalling that $p_{1\to \ell}$ is the probability that one reaches state $\ell$ before state $0$ given initial state $1$.  Equivalently, in our setting, we may interpret $p_{1\to \ell}$ to be the probability that one reaches state $\ell$ before leaving $A$, given initial state $1$, as $0$ is the only state not in $A$.  The claim in Equation \eqref{eq:probclear} follows from conditional expectation and the fact that given that we start in state $1$, we either
\begin{itemize}
\item reach state $\ell$ before leaving $A$, in which case the the expected cumulative time spent in state $j$ before leaving $A$ is $\rew \ell j$ (note that no time is spent in $j$ before reaching $\ell$, as $\ell\le j$),
\item or we do not reach state $\ell$ before leaving $A$, in which case we also do not reach state $j$, and hence we spend 0 time in state $j$ before leaving $A$.
\end{itemize}

From Lemma \ref{lem:pab}, we know that for $\ell\le j$, we have \begin{align}\label{eq:pc}p_{\ell\to j}&=\frac{(\rho\phi(\alpha))^{j-\ell}\left(1-(\rho\phi(\alpha)^2)^\ell\right)}{1-(\rho\phi(\alpha)^2)^j}.\end{align}  Hence, in order to determine $\rew \ell j$ from Equation \eqref{eq:probclear}, we need only determine $\rew 1j$.  We compute this quantity via the renewal reward theorem.  Let us earn reward in state $j$ at rate 1, and consider a cycle from state 0 until one returns to 0 again (after leaving 0).  We also use the fact from Lemma \ref{lem:limprob} that the limiting probability of being in state $j$ in an M/M/1/clearing model is given by $(1-\rho\phi(\alpha))(\rho\phi(\alpha))^j$.  Hence, by the renewal reward theorem, we have
\begin{align}\label{eq:renewclear}
\frac{\rew1j}{\ep[B_C]+1/\lambda}=(1-\rho\phi(\alpha))(\rho\phi(\alpha))^j,
\end{align}
where $B_C$ denotes the busy period of an M/M/1/clearing model.  To determine $\ep[B_C]$, observe that $B_C=\min\{B,\zeta_\alpha\}$, where $B$ is an independent random variable distributed like the busy period of an M/M/1 model \emph{without} clearing, and $\zeta_\alpha\sim\ed(\alpha)$ is an exponentially distributed clearing time.  Taking the expectation, we have
\begin{align*}
\ep[B_C]&=\ep[\min(B,\zeta_\alpha)]=\int_0^\infty\mathbb P(B> t)\mathbb P(\zeta_\alpha>t)\,dt=\int_0^\infty\mathbb P(B\ge t)e^{-\alpha t}\,dt\\
&=\frac1{\alpha}\int_0^\infty\mathbb P(B\ge t)\left(\alpha e^{-\alpha t}\right)\,dt=\frac{\mathbb P(B>\zeta_\alpha)}{\alpha}=\frac{1-\mathbb P(B\le\zeta_\alpha)}{\alpha}\\
&=\frac{1-\phi(\alpha)}{\alpha},
\end{align*}
where the final step follows from an alternate interpretation of the Laplace transform (see Appendix \ref{app:alternate} for details), noting that $\phi(\cdot)$ is the Laplace transform of $B$.

Returning to Equation \eqref{eq:renewclear}, we find that
\begin{align}\label{eq:simple}
\rew 1j&=\left(\frac{1-\phi(\alpha)}{\alpha}+\frac1\lambda\right)(1-\rho\phi(\alpha))(\rho\phi(\alpha))^j=\frac{(\rho\phi(\alpha))^j}{\lambda},
\end{align}
where we make use of the identity \[\left(\frac{1-\phi(\alpha)}{\alpha}+\frac1\lambda\right)(1-\rho\phi(\alpha))=\frac1\lambda\] in our simplification.  This identity can be verified algebraically by using the explicit form of $\phi(s)$.  Alternatively, let $\ep_0[T_0]$ be the expected duration of time spent in state $0$ in a cycle starting from state 0, and ending with a return to state 0 from a nonzero state.  Then by the renewal reward theorem, \[\ep_0[T_0]=\left(\ep[B_C]+\frac1\lambda\right)(1-\rho\phi(\alpha))=\left(\frac{1-\phi(\alpha)}\alpha+\frac1\lambda\right)(1-\rho\phi(\alpha)).\] We can also observe that during such a cycle, the only time spent in state 0 is during the initial residence, as a revisit to state 0 ends the cycle, so $\ep_0[T_0]=1/\lambda$.  Setting these quantities equal to one another  yields the claimed identity directly.

We proceed to use Equation \eqref{eq:probclear} in determining $\rew \ell j$ (in the case where $\ell \le j$), by substituting in values from Equations \eqref{eq:pc} and \eqref{eq:simple}:
\begin{align*}
\rew \ell j&=\frac{\rew1j}{p_{1\to \ell}}=\left(\frac{(\rho\phi(\alpha))^j}{\lambda}\right)\left(\frac{1-(\rho\phi(\alpha)^2)^\ell}{(\rho\phi(\alpha))^{\ell-1}(1-\rho\phi(\alpha)^2)}\right)\\&=\frac{(\rho\phi(\alpha))^{j-\ell+1}\left(1-(\rho\phi(\alpha)^2)^\ell\right)}{\lambda(1-\rho\phi(\alpha)^2)}.
\end{align*}

Next, we consider the case where $\ell\ge j$ (note that the two branches in the claimed expression coincide when $\ell=j$).  We again use conditional expectation, this time obtaining
\begin{align*}
\rew \ell j&=\left(p_{\ell \to j}\right)\rew jj=\left(\phi(\alpha)^{\ell-j}\right)\left(\frac{\rho\phi(\alpha)(1-(\rho\phi(\alpha)^2)^j)}{\lambda(1-\rho\phi(\alpha)^2)}\right)\\&=\frac{\rho\phi(\alpha)^{\ell-j+1}(1-(\rho\phi(\alpha)^2)^j)}{\lambda(1-\rho\phi(\alpha)^2)},
\end{align*}
which completes the proof of the claim.  Note that we have obtained $\rew jj$ by substituting $\ell=j$ into the expression for $\rew \ell j$, which we found for $\ell\le j$, and we have also used the fact from Lemma \ref{lem:pab} that $p_{\ell\to j}=\phi(\alpha)^{\ell-j}$ whenever $\ell\ge j$.
\end{proof}

Finally, we use Lemma \ref{lem:eij} to prove Theorem \ref{thm:eij}.
\begin{theorem:eij}
For any class $\mathbb M$ Markov chain, if $\lambda_m,\mu_m>0$ and $\ell,j\ge j_0+1$, we have
\begin{align}
\label{eq:thm2}
\rep \ell jm&=\begin{cases}
\Omega_m r_m^{j-\ell}\left(1-(r_m\phi_m(\alpha_m))^{\ell-j_0}\right)&\mbox{if }\ell\le j\\
\Omega_{m}\phi_{m}(\alpha_{m})^{\ell - j}\left(1 - (r_{m}\phi_{m}(\alpha_{m}))^{j - j_{0}}\right)&\mbox{if }\ell\ge j.
\end{cases}
\end{align}
\end{theorem:eij}
\begin{proof}
Observe that the time spent in state $(m,j)$ before leaving phase $m$, given initial state $(m,\ell)$ in a class $\mathbb M$ Markov chain with  $\lambda_m,\mu_m>0$ is stochastically identical to the time spent in state $j-j_0$ before reaching state $0$, given initial state $\ell-j_0$ in an M/M/1/clearing model with arrival, departure, and clearing rates $\lambda_m$, $\mu_m$, and $\alpha_m$, respectively.  That is, \[\rep \ell j m=\rew \ell j,\] where the quantity on the left-hand side is associated with the class $\mathbb M$ Markov chain, and the quantity on the right-hand side with the M/M/1/clear\-ing model Markov chain (with the appropriate transition rate parameters and $A=\{1,2,3,\ldots\}$).

We proceed to complete the proof by applying Lemma \ref{lem:eij}.  Recall when $\lambda_m,\mu_m>0$, we have the notation $\rho_m=\lambda_m/\mu_m$, $r_m=\rho_m\phi_m(\alpha_m)$, and $\Omega_m=r_m/(\lambda_m(1-r_m\phi_m(\alpha_m)))$.  Applying Lemma \ref{lem:eij} when $\ell\le j$ yields
\begin{align*}
\rep \ell j m&=\frac{(\rho_m\phi_m(\alpha_m))^{(j-j_0)-(\ell-j_0)+1}\left(1-(\rho_m\phi_m(\alpha_m)^2)^{\ell-j_0}\right)}{\lambda_m(1-\rho_m\phi_m(\alpha_m)^2)}\\
&=\frac{r_m^{j-\ell+1}\left(1-(r_m\phi_m(\alpha_m))^{\ell-j_0}\right)}{\lambda_m(1-r_m\phi_m(\alpha_m))}\\
&=\Omega_m r_m^{j-\ell}\left(1-(r_m\phi_m(\alpha_m))^{\ell-j_0}\right),
\end{align*}
while when $\ell\ge j$, Lemma \ref{lem:eij} yields
\begin{align*}
\rep \ell j m&=\frac{\rho_m\phi_m(\alpha_m)^{(\ell-j_0)-(j-j_0)+1}\left(1-(\rho_m\phi_m(\alpha_m)^2)^{j-j_0}\right)}{\lambda_m(1-\rho_m\phi_m(\alpha_m)^2)}\\
&=\frac{r_m\phi_m(\alpha_m)^{\ell-j}\left(1-(r_m\phi_m(\alpha_m))^{j-j_0}\right)}{\lambda_m(1-r_m\phi_m(\alpha_m))}\\
&=\Omega_m \phi_m(\alpha_m)^{\ell-j}\left(1-(r_m\phi_m(\alpha_m))^{j-j_0}\right),
\end{align*}
as claimed.\end{proof}

\section{Extending the scope of the CAP Method}\label{sec:scope}

In this section we briefly touch upon ways in which the CAP method can be extended beyond class $\mathbb M$ Markov chains.

\subsection{Chains with ``catastrophes''}

Recall that the M/M/1 clearing model is used to model a system where there can be a catastrophe from any nonzero state causing an immediate transition to state $0$.  Similarly, we can consider a modification of a class $\mathbb M$ Markov chain where from any state $(m,j)$ with $j\ge j_0+1$, a catastrophe can occur taking one to state $x\in\mathcal N$ with rate $\alpha_m\langle x\rangle\equiv q((m,j),x)$.\footnote{Whether or not catastrophes can also occur in states $(m,j_0)$ will not change the analysis as arbitrary transitions from states $(m,j_0)$ to states $x\in\mathcal N$ are already allowed in class $\mathbb M$ Markov chains.}  That is, each phase can have several catastrophe rates, one for each state in the non-repeating portion.  In this case, it will be useful to redefine $\alpha_m$ as follows: \[\alpha_m\equiv\sum_{x\in\mathcal N} \alpha_m\langle x\rangle+\sum_{i=m+1}^M\sum_{\Delta=-1}^1\alp m{i-m}{\Delta}.\]  The CAP method can easily be modified to give limiting probabilities for these types of Markov chains.

\subsection{Skipping levels when transitioning between phases}

Although the assumption that transitions from state $(m,j)$ to state $(m,\ell)$ can only occur only if $\ell=j\pm 1$ is essential to the CAP method, the assumption that transitions from state $(m,j)$ to state $(i,\ell)$ (where $i>m$) can only occur if $\ell=j\pm1$ is much less important.  That is, the CAP method may be extended to allow for nonzero transition rates of the form $\alp m{\Delta_1}{\Delta_2}$ with $d\le \Delta_2\le D$ for some $d,D\in\mathbb Z$.  However, it is advisable to treat the levels $L_{j_0},L_{j_{0}+1},\ldots,L_{j_{0}+\max\{|d|,|D|\}-1}$ as special cases, just as $L_{j_0}$ was treated as a special case in the analysis presented throughout this paper.

\subsection{Chains with an infinite number of phases}

Consider a chain with the structure of a class $\mathbb M$ chain, except with infinitely many \emph{phases} (i.e., $m\in\{0,1,2,\ldots\}$), and a possibly infinite non-repeating portion, $\mathcal N$.  The CAP method may be used to determine the $\{c_{m,k}\}_{0\le k\le m}$ values in terms of $\{\pi_x\}_{x\in\mathcal N}$ for the first $K$ phases by solving a system of at most $O(K^2)$ equations.  This is because the CAP method provides recurrences such that each $\{c_{m,k}\}_{0\le k\le m}$ value can be expressed in terms of $\{c_{i,k}\}_{0\le k\le i\le m-1}$ values; that is, only information about lower-numbered phases (and the non-repeating portion) is needed to compute each $c_{m,k}$.  We can first express such values for phase $m=0$, then phase $m=1$, and so on.  Once these values---along with the easily determined corresponding base terms---have been obtained, we can use the CAP method to find the limiting probabilities for all states in the first $K$ phases as long as we know the $\{\pi_x\}_{x\in\mathcal N}$ values.

Such a procedure is typically not useful, as the $\{\pi_x\}_{x\in\mathcal N}$ values are usually determined via the normalization constraint, which requires expressing limiting probabilities, $\pi_{(m,j)}$, in terms of $\{\pi_x\}_{x\in\mathcal N}$ for \emph{all} phases, rather than for only the first $K$ phases.  However, there are settings where sufficient information about the structure of $\{\pi_x\}_{x\in\mathcal N}$ may be obtained via other analytic approaches, allowing for the CAP method to compute the limiting probability of the first $K$ phases (where $K$ can be as high as desired, subject to computational constraints).  For example, a two-class priority queue can be modeled by an infinite phase variant of a class $\mathbb M$ Markov chain.  In that setting, queueing-theoretic analysis provides sufficient information about the structure of the limiting probabilities in the non-repeating portion (see \cite{sleptchenko2014joint}), making the CAP method a useful tool for that problem.

\section{Conclusion}

This paper presents a study of the stationary distribution of quasi-birth-death (QBD) continuous time Markov chains in class $\mathbb M$.  Class $\mathbb M$ Markov chains are ergodic chains consisting of a finite nonrepeating portion and an infinite repeating portion.  The repeating portion of a class $\mathbb M$ chain consists of an infinite number of levels and a finite number of phases.  Moreover, transitions in such chains are \emph{skip-free in level}, in that one can only transition between consecutive levels, and \emph{unidirectional in phase}, in that one can only transition from lower-numbered phases to higher-numbered phases.  Despite these restrictions, class $\mathbb M$ Markov chains are used extensively in modeling computing, service, and manufacturing systems, as they allow for keeping track of both the number of jobs in a system (via levels), and the state of the server(s) and/or the arrival process to the system (via phases).

This paper develops and introduces a novel technique, Clearing Analysis on Phases (CAP), for determining the limiting probabilities of class $\mathbb M$ chains exactly.  This method proceeds iteratively among the phases, by first determining the form of the limiting probabilities of the states in phase 0, then proceeding to do the same for the states in phase 1, and so on.  As suggested by its name, the CAP method uses clearing model analysis to determine the structure of the limiting probabilities in each phase.

Unlike most existing techniques for solving QBDs, which rely upon the matrix-geometric approach, the CAP method avoids the task of finding the complete rate matrix, $\mathbf R$, entirely.  Instead, the CAP method yields the limiting probabilities of each state, $(m,j)$, in the repeating portion of the Markov chain as a linear combination of scalar \emph{base terms} (with weights dependent on the phase, $m$), each raised to a power corresponding to the level, $j$.  These \emph{base terms} turn out to be the diagonal elements of the rate matrix, $\mathbf R$.  The weights of these linear combinations can be determined by solving a finite system of linear equations.  We also observe that the structure of the weights of these linear combinations can depend on the multiplicity structure of the base terms.

The CAP method can be applied to Markov chains beyond those in class $\mathbb M$, as discussed in Section \ref{sec:scope}.  For example, the CAP method can be used to determine limiting probabilities in chains where one or more phases allow for immediate ``catastrophe'' transitions to states in the non-repeating portion.  As another example, the CAP method can also be applied to Markov chains where transitions between phases can be accompanied with a change in level exceeding 1. The CAP method can also be used to study some chains with an infinite number of phases.  There is ample room for future work to extend the CAP method  in a variety of directions.

The CAP method and the solution form it provides offer several impactful advantages.  First, while many existing methods for determining the limiting probabilities of QBDs exploit the relationship between successive \emph{levels}, the CAP method exploits the relationship between successive \emph{phases}, thereby offering complementary probabilistic intuition on the structure and steady-state behavior of class $\mathbb M$ Markov chains.  This method also provides an additional tool for practitioners who are studying systems that can be modeled by class $\mathbb M$ Markov chains.  Depending on the application domain, the scalar solution form of the CAP method may have advantages over other solution forms for computing certain metrics of interest (e.g., mean values, higher moments, tail probabilities, etc.).  While this paper does not cover using the solution of the CAP method to derive metrics of interest, as such metrics are often application specific, we hope that future work can find novel uses for the CAP method in a variety of settings.

\bibliographystyle{abbrv}
{\bibliography{limitingprob.bib}}

\clearpage

\appendix

\bigskip

\noindent{\Large\textbf{Appendix}}






\section{An alternative interpretation of the Laplace transform}\label{app:alternate}

Let $X$ be a nonnegative random variable, with well-defined Laplace transform $\psi(\cdot)$ (i.e., $\psi$ is defined on all positive reals), cumulative distribution function, $F_X(\cdot)$, and probability density function, $f_X(\cdot)$; note that $X$ may have nonzero probability mass at $+\infty$, in which case $\int_0^\infty f_X(t)\,dt<1$ (where we interpret the integral as being evaluated on $\{t\in\mathbb R\colon 0\le t<\infty\}$).  Then for any constant $w>0$, we have the following interpretation of $\psi$:
\begin{align*}
\psi(w)&=\int_{0}^\infty e^{-wt}f_{X}(t)\,dt\\
&=\left.e^{-wt}F_X(t)\right|_0^\infty+\int_0^\infty F_X(t)\left(we^{-wt}\right)\,dt\\
&=\mathbb P\{X\le\zeta_w\},
\end{align*}
where $\zeta_w\sim\ed(w)$ is a random variable independent of $X$.

\section{The Complete Proof of Theorem \ref{THM:RESULT}}\label{app:complete}
\begin{proof}


We prove the theorem via strong induction on the phase, $m$.  Specifically, for each phase $m$, we will show that $\pi_{(m,j)}$ takes the form $\pi_{(m,j)}=\sum_{k=0}^m c_{m,k}r_k^{j-j_0}$ for all $j\ge j_0+1$, and show that $\{c_{m,k}\}_{0\le k\le m-1}$ satisfies
{\footnotesize \[c_{m,k}=\begin{cases}\displaystyle{\frac{r_{k}r_m}{\lambda_m(r_k-r_m)(1-\phi_m(\alpha_m)r_k)}\left(\sum_{i=k}^{m-1}\sum_{\Delta=-1}^1c_{i,k}\alp i{m-i}\Delta r_k^\Delta\right)}&\mbox{if }r_m,r_k>0\\
\displaystyle{\frac{\displaystyle{\sum_{i=k}^{m-1}\sum_{\Delta=-1}^1c_{i,k}\alp i{m-i}\Delta r_k^\Delta}}{\mu_m(1-r_k)+\alpha_m}}&\mbox{if }r_k>r_m=0\\
0&\mbox{if }r_k=0,
\end{cases}\] }
while $c_{m,m}=\pi_0-\sum_{k=0}^{m-1} c_{m,k}$.
Finally, after completing the inductive proof, we justify that the remaining linear equations in the proposed system are ordinary balance equations together with the normalization constraint.

\smallskip

\noindent\textbf{Base case:}

We begin our strong induction by verifying that the claim holds for the base case (i.e., for $m=0$).  By the ergodicity requirement on class $\mathbb M$ Markov chains, $\lambda_0>0$, leaving two sub-cases when $m=0$: the case where $\mu_0>0$, and the case where $\mu_0=0$.  In the first case, where $\mu_0>0$, Equation \eqref{eq:thm} yields \[\rep {j_0+1}j0=\Omega_0 r_0^{j-j_0-1}(1-r_0\phi_0(\alpha_0))=\frac{r_0^{j-j_0}}{\lambda_0}.\]

Now consider the other sub-case, where $\mu_0=0$, recalling that in this case, we have $r_0=\lambda_0/(\lambda_0+\alpha_0)$.  We calculate $\rep {j_0+1}j0$ for this case, by noting that transitions within states in $P_0$ cannot decrease the level, as follows: starting at state $(0,j_0+1)$, we either never visit state $(0,j)$ before leaving $P_0$, or we visit state $(0,j)$ \emph{exactly once} before leaving $P_0$. The latter occurs with probability \[\left(\frac{\lambda_0}{\lambda_0+\alpha_0}\right)^{j-j_0-1}=r_0^{j-j_0-1},\] in which case, we spend an average of $1/(\lambda_0+\alpha_0)=r_0/\lambda_0$ units of time in state $(0,j)$.  Hence, we find that \[\rep {j_0+1}j0=r_0^{j-j_0-1}\left(\frac{r_0}{\lambda_0}\right)=\frac{r_0^j}{\lambda_0},\] which coincides with our finding for the case where $\mu_0>0$.

In both cases, applying Theorem \ref{thm:main} yields
\begin{align*}
\pi_{(0,j)}&=\pi_{(0,j_0)}\lambda_0\rep {j_0+1}j0=\pi_{(0,j_0)}\lambda_0\left(\frac{r_0^{j-j_0}}{\lambda_0}\right)=\pi_{(0,j_0)}r_0^{j-j_0}\\
&=c_{0,0}r_0^{j-j_0},
\end{align*} where $c_{0,0}=\pi_{(0,j_0)}$.  Hence, $\pi_{(0,j)}$ takes the claimed form. Moreover, $c_{0,0}$ satisfies the claimed constraint as $c_{0,0}=\pi_{(0,j_0)}-\sum_{k=0}^{m-1} c_{m,k}=\pi_{(0,j_0)}-0=\pi_{(0,j_0)}$, because the sum is empty when $m=0$.  Note that when $m=0$, $\{c_{m,k}\}_{0\le k<m\le M}$ is empty, and hence, there are no constraints on these values that require verification.


\smallskip

\noindent\textbf{Inductive step:}

Next, we proceed to the inductive step and assume the induction hypothesis holds for all phases $i\in\{0,1,\ldots,m-1\}$.  In particular, we assume that $\pi_{(i,j)}=\sum_{k=0}^i c_{i,k} r_k^{j-j_0}$ for all $i<m$.  For convenience, we introduce the notation
\[\Upsilon_{m,j}\equiv\lambda_m \rep{j_0+1}jm\quad\mbox{and}\quad\Psi_{m,k,j}\equiv\sum_{\ell=1}^\infty r_k^{\ell-j_0}\rep \ell j m.\]
Using this notation, we apply Theorem \ref{thm:main} and the induction hypothesis, which yields\footnote{
Note that $\sum_{\Delta=-1}^1\alp i{m-i}\Delta r_k^\Delta$ is not well-defined when $r_k=0$, as $0^{-1}$ and $0^0$ are not well-defined.  However, this is just a convenient formal manipulation which will remain true if we assign any real value to $\sum_{\Delta=-1}^1\alp i{m-i}\Delta r_k^\Delta$ as $\Psi_{m,k,j}=0$ in the $r_k=0$ case, and the ``contribution'' to the sum by an index $k$ such that $r_k=0$ is also 0.  One can verify that this is ``harmless'' by examining such $k$ indices in isolation.  Note further that we have also used the fact that $\pi_{(i,j_0)}$ also satisfies the claimed form for all $i<m$, which is true as $c_{i,i}=\pi_{(i,j_0)}-\sum_{k=0}^{i-1} c_{i,k}$ (from the inductive hypothesis) implies that $\pi_{(i,j_0)}=\sum_{k=0}^{i} c_{i,k}=\sum_{k=0}^{i} c_{i,k}r_k^0$, except that once again values of $r_k=0$ yield undefined quantities of the form $0^0$.  Once again, this is a convenient formal manipulation that will not affect our results if we simply assign $0^0=1$ in this context.}
{\footnotesize
\begin{align}
\pi_{(m,j)}&=\pi_{(m,j_0)}\lambda_m\rep{j_0+1}jm+\sum_{i=0}^{m-1}\sum_{\ell=1}^\infty\sum_{\Delta=-1}^1\pi_{(i,\ell-\Delta)}\alp i{m-i}\Delta\rep {\ell} j m\nonumber\\
&=\pi_{(m,j_0)}\Upsilon_{m,j}+\sum_{i=0}^{m-1}\sum_{\ell=1}^\infty\sum_{\Delta=-1}^1\alp i{m-i}\Delta\left(\sum_{k=0}^{i}c_{i,k}r_k^{\ell-j_0-\Delta}\rep \ell j m\right)\nonumber\\
&=\pi_{(m,j_0)}\Upsilon_{m,j}+\sum_{k=0}^{m-1}\sum_{i=k}^{m-1}\left(c_{i,k}\sum_{\Delta=-1}^1\alp i{m-i}\Delta r_k^\Delta\right)\left(\sum_{\ell=1}^\infty r_k^{\ell-j_0}\rep \ell j m\right)\nonumber\\
&=\pi_{(m,j_0)}\Upsilon_{m,j}+\sum_{k=0}^{m-1}\sum_{i=k}^{m-1}\left(c_{i,k}\sum_{\Delta=-1}^1\alp i{m-i}\Delta r_k^\Delta\right)\Psi_{m,k,j} \label{eq:main2}.
\end{align}
}
We proceed to compute $\Upsilon_{m,j}$ and $\Psi_{m,k,j}$ separately in the following cases:
\begin{itemize}
\item \textbf{Case 1}: $\lambda_m,\mu_m>0$
\item \textbf{Case 2}: $\lambda_m>\mu_m=0$
\item \textbf{Case 3}: $\mu_m>\lambda_m=0$
\item \textbf{Case 4}: $\mu_m=\lambda_m=0$
\end{itemize}

\smallskip

\noindent\textbf{Computations for Case 1 ($\lambda_m,\mu_m>0$):}

When $\lambda_m,\mu_m>0$, Equation \eqref{eq:thm} yields $\Upsilon_{m,j}=\lambda_m\rep {j_0+1}jm=r_m^{j-j_0}.$ We also find that
\begin{align*}
\Psi_{m,k,j}&=\sum_{\ell=1}^\infty r_k^{\ell-j_0}\rep\ell jm\\
&= \sum_{\ell=j_0+1}^j r_k^{\ell-j_0}\rep \ell j m+\sum_{\ell=j+1}^\infty r_k^{\ell-j_0}\rep\ell jm\\
&=\Omega_m\left(\sum_{\ell=j_0+1}^j r_k^{\ell-j_0}r_m^{j-\ell}\left(1-(r_m\phi_m(\alpha_m))^{\ell-j_0}\right)\right.\\
&\qquad\qquad\left.+\sum_{\ell=j+1}^\infty r_k^{\ell-j_0}\phi_{m}(\alpha_m)^{\ell-j}\left(1-(r_m\phi_m(\alpha_m))^{j-j_0}\right)\right)\\
&=\frac{r_{k}r_m(r_k^{j-j_0}-r_m^{j-j_0})}{\lambda_m(r_k-r_m)(1-\phi_m(\alpha_m)r_k)},
\end{align*}
where the last equality follows from well known geometric sum identities.  Note that this expression is well-defined because $r_k\neq r_m$ by assumption and $r_m\phi_m(\alpha_m)\neq1$.

\smallskip

\noindent\textbf{Computations for Case 2 ($\lambda_m>\mu_m=0$):}

When $\lambda_m>\mu_m=0$, we recall that $r_m=\lambda_m/(\lambda_m+\alpha_m)$ and compute $\rep \ell j m$ as follows: starting at state $(m,\ell)$, we either never visit state $(m,j)$ before leaving $P_m$, or we visit state $(m,j)$ \emph{exactly once} before leaving $P_m$.  If $\ell>j$, we never visit state $(m,j)$ before leaving $P_m$ (and so $\rep \ell jm=0$), but if $\ell\le j$, we visit state $(m,j)$ \emph{exactly once} before leaving $P_m$ with probability $r_m^{j-\ell}$, and this visit will last an average time of $1/(\lambda_m+\alpha_m)=r_m/\lambda_m$, yielding
\[\rep \ell jm=r_m^{j-\ell}\left(\frac{r_m}{\lambda_m}\right)=\frac{r_m^{j-\ell+1}}{\lambda_m}.\]  In particular, $\Upsilon_{m,j}=\lambda_m\rep {j_0+1}jm=r_m^{j-j_0},$ coinciding with the expression for $\Upsilon_{m,j}$ from Case 1, and furthermore, we have
{
\begin{align*}
\Psi_{m,k,j}&=\sum_{\ell=j_0+1}^\infty r_k^{\ell-j_0}\rep\ell jm\\
&= \sum_{\ell=j_0+1}^{j} r_k^{\ell-j_0}\rep \ell j m+\sum_{\ell=j+1}^\infty r_k^{\ell-j_0}\rep\ell jm\\
&=\sum_{\ell=j_0+1}^{j} \frac{r_k^{\ell-j_0} r_m^{j-\ell+1}}{\lambda_m}\\
&=\frac{r_k r_m(r_k^{j-j_0}-r_m^{j-j_0})}{\lambda_m(r_k-r_m)}=\frac{r_k r_m(r_k^{j-j_0}-r_m^{j-j_0})}{\lambda_m(r_k-r_m)(1-\phi_m(\alpha_m)r_k)}.
\end{align*}}
which coincides with the expression for $\Psi_{m,k,j}$ that we found in Case 1.  The last equality follows by noting that in this case we have $\phi_m(s)\equiv0$, and hence $1-\phi_m(\alpha_m)r_k=1$.

\smallskip

\noindent\textbf{Computations for Case 3 ($\mu_m>\lambda_m=0$):}

When $\mu_m>\lambda_m=0$, we  have $\Upsilon_{m,j}=\lambda_m\rep {j_0+1}jm=0.$  Next, we compute $\rep\ell jm$ as follows: starting at state $(m,\ell)$, if $\ell<j$, we never visit $j$ before leaving $P_m$, while if $\ell\ge j$ we will visit $j$ {exactly once} with probability $\mu_m^{\ell-j}/(\mu_m+\alpha_m)^{\ell-j}$ and this visit will last an average duration of $1/(\mu_m+\alpha_m)$ units of time.  Consequently, $\rep\ell jm=0$ in the former case and \[\rep\ell jm=\frac{\mu_m^{\ell-j}}{(\mu_m+\alpha_m)^{\ell-j+1}}\] in the latter case.  Finally, we have
\begin{align*}
\Psi_{m,k,j}&=\sum_{\ell=j_0+1}^\infty  r_k^{\ell-j_0}\rep\ell jm\\
&= \sum_{\ell=j_0+1}^{j-1}  r_k^{\ell-j_0}\rep \ell j m+\sum_{\ell=j}^\infty  r_k^{\ell-j_0}\rep\ell jm\\
&=\sum_{\ell=j}^\infty\frac{r_k^{\ell-j_0}\mu_m^{\ell-j}}{(\mu_m+\alpha_m)^{\ell-j+1}}=\frac{r_k^{j-j_0}}{\mu_m(1-r_k)+\alpha_m}.
\end{align*}

\smallskip

\noindent\textbf{Computations for Case 4 ($\mu_m=\lambda_m=0$):}

When $\mu_m=\lambda_m=0$, we again have $\Upsilon_{m,j}=\lambda_m\rep {j_0+1}jm=0,$ as in Case 3.  Next, we compute $\rep\ell jm$ as follows: in this case any visit to $P_m$ will consist entirely of one visit to the initial state in $P_m$, as there are no transitions to other states in the same phase.  Hence, $\rep\ell j m=\alpha_m$ if $\ell=j$, and $\rep\ell j m=0$ otherwise.  Consequently,
\begin{align*}
\Psi_{m,k,j}&=\sum_{\ell=1}^\infty r_k^{\ell-j_0} \rep\ell j m=r_k^{j-j_0}\rep jjm=\frac{r_k^{j-j_0}}{\alpha_m}\\
&=\frac{r_k^{j-j_0}}{\mu_m(1-r_k)+\alpha_m},
\end{align*}
which coincides with the expression for $\Psi_{m,k,j}$ that we found in Case 3. The last equality follows by noting that $\mu_m=0$, and hence $\mu_m(1-r_k)=0$.

\smallskip

\noindent\textbf{Completing the inductive step:}

We now proceed to substitute the results of our computations into Equation \eqref{eq:main2}.  Since $\Upsilon_{m,j}$ can be given by the same expression for both Case 1 and 2, and the same holds for $\Psi_{m,k,j}$, we consider these two cases together, and note that they jointly make up the case where $r_m>0$.  For $j\ge j_0+1$,
{\footnotesize
\begin{align*}
\pi_{(m,j)}&=\pi_{(m,j_0)}\Upsilon_{m,j}+\sum_{k=0}^{m-1}\sum_{i=k}^{m-1}\left(c_{i,k}\sum_{\Delta=-1}^1\alp i{m-i}\Delta r_k^\Delta\right)\Psi_{m,k,j}\\
&=\pi_{(m,j_0)}r_m^{j-j_0}+\sum_{k=0}^{m-1}\sum_{i=k}^{m-1}\left(c_{i,k}\sum_{\Delta=-1}^1\alp i{m-i}\Delta r_k^\Delta\right)\left(\frac{r_k r_m(r_k^{j-j_0}-r_m^{j-j_0})}{\lambda_m(r_k-r_m)(1-\phi_m(\alpha_m)r_k)}\right)\\
&=\sum_{k=0}^m c_{m,k}r_k^{j-j_0},
\end{align*}}
where we have collected terms with
{\small
\begin{align*}
c_{m,k}&=\frac{\displaystyle{r_kr_m\left(\sum_{i=k}^{m-1}\sum_{\Delta=-1}^1c_{i,k}\alp i{m-i}\Delta r_k^\Delta\right)}}{\lambda_m(r_k-r_m)(1-\phi_m(\alpha_m)r_k)}& (&0\le k < m\le M\colon r_m,r_k>0)
\end{align*}
}
and $c_{m,k}=0$ when $r_m>r_k=0$ and $c_{m,m}=\pi_{(m,j_0)}-\sum_{k=0}^{m-1}c_{m,k}$, as claimed.

The expressions for $\Upsilon_{m,j}$ and $\Psi_{m,k,j}$ also coincide across Cases 3 and 4 (although they are distinct from their Case 1 and 2 counterparts), so we also consider these two cases together, noting that they jointly make up the case where $\lambda_m=r_m=0$:
\begin{align*}
\pi_{(m,j)}&=\pi_{(m,j_0)}\Upsilon_{m,j}+\sum_{k=0}^{m-1}\sum_{i=k}^{m-1}\left(c_{i,k}\sum_{\Delta=-1}^1\alp i{m-i}\Delta r_k^\Delta\right)\Psi_{m,k,j}\\
&=0+\sum_{k=0}^{m-1}\sum_{i=k}^{m-1}\left(c_{i,k}\sum_{\Delta=-1}^1\alp i{m-i}\Delta r_k^\Delta\right)\left(\frac{r_k^{j-j_0}}{\mu_m(1-r_k)+\alpha_m}\right)\\
&=\sum_{k=0}^m c_{m,k}r_k^{j-j_0}
\end{align*}
where we have collected terms with
\begin{align*}
c_{m,k}&=\frac{\displaystyle{\sum_{i=k}^{m-1}\sum_{\Delta=-1}^1c_{i,k}\alp i{m-i}\Delta r_k^\Delta}}{\mu_m(1-r_k)+\alpha_m}&(&0\le k < m\le M\colon r_m,r_k>0)
\end{align*} and $c_{m,k}=0$ when $r_m=r_k=0$.  Observe that since $r_m=0$, it appears that we can allow $c_{m,m}$ to take any real value, so in order to satisfy the induction hypothesis, we set $c_{m,m}=\pi_{(m,j_0)}-\sum_{k=0}^{m-1}c_{m,k}$ in the $r_m=0$ case as well.  Also note that we have set $c_{m,k}=0$ when $r_k=0$ in both the $r_m>0$ and $r_m=0$ cases.  This completes the inductive step and the proof by induction.

\smallskip

\noindent\textbf{The balance equations and normalization constraint:}

The equations with $\pi_{(m,j_0)}$ and $\pi_x$ in their left-hand sides in our proposed system are ordinary balance equations (that have been normalized so that there are no coefficients on the left-hand side).

It remains to verify that the final equation, which is the normalization constraint:
\begin{align*}
1&=\sum_{x\in\mathcal N}\pi_x+\sum_{m=0}^M\pi_{(m,j_0)}+\sum_{m=0}^M\sum_{j=j_0+1}^\infty \pi_{(m,j)}\\
&=\sum_{x\in\mathcal N}\pi_x+\sum_{m=0}^M\sum_{k=0}^M c_{m,k}+\sum_{m=0}^M\sum_{k=0}^{m-1}\sum_{j=j_0+1}^\infty c_{m,k}r_k^{j-j_0}\\
&=\sum_{x\in\mathcal N}\pi_x+\sum_{m=0}^M\sum_{k=0}^{m} \frac{c_{m,k} r_k}{1-r_k}.
\end{align*}
\end{proof}

\section{Negative Binomial Lemmas}\label{app:negbinlems}

These lemmas are used to derive our main results, and are likely known, but to make the paper self-contained we both state and prove them.

\begin{lemma} \label{negbinlemmaone} For each $\beta \in (0,1)$, we have
\begin{align*}
\sum_{\ell = j_{0}}^{\infty}{\ell - j_{0} + n \choose n}\beta^{\ell - (j_{0} - 1)} = \frac{\beta}{(1 - \beta)^{n+1}}.
\end{align*}
\end{lemma}
\begin{proof} Having a negative binomial distribution with parameters $n + 1$ and $(1 - \beta)$ in mind, we observe that
{\footnotesize
\begin{align*}
 \sum_{\ell = j_{0}}^{\infty}{\ell - j_{0} + n \choose n}\beta^{\ell - (j_{0} - 1)} &= \beta \sum_{\ell = j_{0}}^{\infty}{\ell - j_{0} + n \choose n}\beta^{\ell - j_{0}} \\
 &= \frac{\beta}{(1 - \beta)^{n+1}}\sum_{\ell = j_{0}}^{\infty}{\ell - j_{0} + n \choose (n+1)-1}\beta^{\ell - j_{0}}(1 - \beta)^{n+1} \\
 &= \frac{\beta}{(1 - \beta)^{n+1}}\sum_{k=0}^{\infty}{(k + n + 1) - 1 \choose (n+1) - 1}\beta^{k}(1 - \beta)^{n+1} \\
 &= \frac{\beta}{(1 - \beta)^{n+1}}\sum_{\ell = n + 1}^{\infty}{\ell - 1 \choose (n+1) - 1}\beta^{\ell - (n+1)}(1 - \beta)^{n+1} \\
 &= \frac{\beta}{(1 - \beta)^{n+1}}.
 \end{align*}
 }
 \end{proof}

The next lemma shows how to compute a truncated version of the above series.

\begin{lemma} \label{negbinlemmatwo} For $\beta \neq 1$, we have
{\small
\begin{align*}
\sum_{\ell = j_{0}}^{j-1}{\ell - j_{0} + n \choose n}\beta^{\ell - (j_{0} - 1)} &= \frac{\beta - \beta^{j - (j_{0} - 1)}}{(1 - \beta)^{n+1}} \\
&\quad- \sum_{k=1}^{n}\left[{j - j_{0} + k \choose k} - {j - j_{0} + k - 1 \choose k-1}\right]\frac{\beta^{j - (j_{0} - 1)}}{(1 - \beta)^{n + 1 - k}}.
\end{align*}
}
\end{lemma}
\begin{proof} Starting with the left-hand-side, we have
{\footnotesize
\begin{align*}
\sum_{\ell = j_{0}}^{j-1}{\ell - j_{0} + n \choose n}\beta^{\ell - (j_{0} - 1)} &= \sum_{\ell = j_{0}}^{j-1}\sum_{x=j_{0}}^{\ell}{x - j_{0} + n - 1 \choose n-1}\beta^{\ell - (j_{0} - 1)} \\
&= \sum_{x = j_{0}}^{j-1}\sum_{\ell = x}^{j - 1}{x - j_{0} + n - 1 \choose n-1}\beta^{\ell - (j_{0} - 1)} \\
&= \sum_{x = j_{0}}^{j-1}{x - j_{0} + n - 1 \choose n-1}\beta^{x - (j_{0} - 1)}\sum_{\ell = x}^{j-1}\beta^{\ell - x} \\
&= \frac{1}{(1 - \beta)}\sum_{x = j_{0}}^{j-1}{x - j_{0} + n - 1 \choose n-1}\beta^{x - (j_{0} - 1)}(1 - \beta^{j-x}) \\
&= \frac{1}{(1 - \beta)}\sum_{x = j_{0}}^{j-1}{x - j_{0} + n - 1 \choose n-1}\beta^{x - (j_{0} - 1)} \\
&\qquad- \frac{1}{1 - \beta}{j - j_{0} + n - 1 \choose n}\beta^{j - (j_{0} - 1)} \\
&= \frac{1}{(1 - \beta)}\sum_{x = j_{0}}^{j-1}{x - j_{0} + n - 1 \choose n-1}\beta^{x - (j_{0} - 1)}\\
&\qquad - \frac{1}{1 - \beta}\left[{j - j_{0} + n \choose n} - {j - j_{0} + n - 1 \choose n-1}\right]\beta^{j - (j_{0} - 1)}.
\end{align*}
}
Setting
\begin{align*} a_{n} &= \sum_{\ell = j_{0}}^{j-1}{\ell - j_{0} + n \choose n}\beta^{\ell - (j_{0} - 1)},\\ b_{n} &= \left[{j - j_{0} + n \choose n} - {j - j_{0} + n - 1 \choose n-1}\right]\beta^{j - (j_{0} - 1)}
\end{align*} we see that for each $n\in\{1,2,3,\ldots\}$ we have
\begin{align*} a_{n} = \frac{a_{n-1}}{1 - \beta} - \frac{b_n}{1 - \beta} \end{align*} where
\begin{align*} a_{0} = \frac{\beta - \beta^{j - (j_{0} - 1)}}{1 - \beta}. \end{align*}  The solution to this recursion is given by
\begin{align*} a_{n} = \frac{a_{0}}{(1 - \beta)^{n}} - \sum_{k=1}^{n}\frac{b_{k}}{(1 - \beta)^{n+1-k}} \end{align*} or, equivalently,
\begin{align*} a_{n} = \frac{1 - \beta^{j - (j_{0} - 1)}}{(1 - \beta)^{n+1}} - \sum_{k=1}^{n}\left[{j - j_{0} + n \choose n} - {j - j_{0} + n - 1 \choose n-1}\right]\frac{\beta^{j - (j_{0} - 1)}}{(1 - \beta)^{n + 1 - k}}, \end{align*} which completes our derivation.  \end{proof}

The next lemma can be viewed as a generalization of Lemma \ref{negbinlemmaone}.

\begin{lemma} \label{negbinlemmathree} For $\beta \in (0,1)$,
{\footnotesize
\begin{align*} \sum_{\ell = j}^{\infty}{\ell - j_{0} + n \choose n} \beta^{\ell - j} = \frac{1}{(1 - \beta)^{n+1}} + \sum_{k=1}^{n}\left[{j - j_{0} + k \choose k} - {j - j_{0} + k - 1 \choose k-1}\right]\frac{1}{(1 -\beta)^{n+1-k}}. \end{align*}
}
\end{lemma}
\begin{proof} The key to deriving this series is to use both Lemmas \ref{negbinlemmaone} and \ref{negbinlemmatwo}.  Here
{\footnotesize
\begin{align*}
\sum_{\ell = j}^{\infty}{\ell - j_{0} + n \choose n} \beta^{\ell - j} &= \beta^{j_{0} - j}\sum_{\ell - j_{0}}^{\infty}{\ell - j_{0} + n \choose n}\beta^{\ell - j_{0}}\\
&\quad - \beta^{(j_{0} - 1) - j}\sum_{\ell = j_{0}}^{j-1}{\ell - j_{0} + n \choose n}\beta^{\ell - (j_{0} - 1)} \\
&= \frac{\beta^{j_{0} - j}}{(1 - \beta)^{n+1}} - \beta^{(j_{0} - 1) - j}\left[\frac{\beta - \beta^{j - (j_{0} - 1)}}{(1 - \beta)^{n+1}}\right] \\
&\quad+ \sum_{k=1}^{n}\left[{j - j_{0} + k \choose k} - {j - j_{0} + k - 1 \choose k-1}\right]\frac{1}{(1 - \beta)^{n + 1 - k}} \\
&= \frac{1}{(1 - \beta)^{n+1}}  + \sum_{k=1}^{n}\left[{j - j_{0} + k \choose k} - {j - j_{0} + k - 1 \choose k-1}\right]\frac{1}{(1 - \beta)^{n + 1 - k}}.
\end{align*}
}
thus proving the claim. \end{proof}

\end{document}

%% file: chain.tikz.tex
\begin{tikzpicture}[scale=2.45]
\draw [ultra thick,|->] (-.8,-3.25) -- (-.8,.5) -- (4,.5);
\node[above] at (1.6,.5) {\large {level, $j$}};
\node[left] at (-.8,-1.375) {\rotatebox{90}{{\large phase, $m$}}};
\node[draw, circle, ultra thick, name=nrp, minimum size=.8in] at (-.25,-1.5) {{\large $\mathcal N$}};
\node[draw, circle, thick, name=11, minimum size=.6in] at (1,0) {{\small $(0,0)$}};
\node[draw, circle, thick, name=12, minimum size=.6in] at (2,0) {{\small $(0,1)$}};
\node[draw, circle, thick, name=13, minimum size=.6in] at (3,0) {{\small $(0,2)$}};
\node[circle, thick, name=1d, minimum size=.6in] at (4,0) {{$\cdots$}};
\node[draw, circle, thick, name=21, minimum size=.6in] at (1,-1) {{\small $(1,0)$}};
\node[draw, circle, thick, name=22, minimum size=.6in] at (2,-1) {{\small $(1,1)$}};
\node[draw, circle, thick, name=23, minimum size=.6in] at (3,-1) {{\small $(1,2)$}};
\node[circle, thick, name=2d, minimum size=.6in] at (4,-1) {{$\cdots$}};
\node[circle, thick, name=d1, minimum size=.6in] at (1,-2) {{$\vdots$}};
\node[circle, thick, name=d2, minimum size=.6in] at (2,-2) {{$\vdots$}};
\node[circle, thick, name=d3, minimum size=.6in] at (3,-2) {{$\vdots$}};
\node[draw, circle, thick, name=k1, minimum size=.6in] at (1,-3) {{\small $(M,0)$}};
\node[draw, circle, thick, name=k2, minimum size=.6in] at (2,-3) {{\small $(M,1)$}};
\node[draw, circle, thick, name=k3, minimum size=.6in] at (3,-3) {{\small $(M,2)$}};
\node[circle, thick, name=kd, minimum size=.6in] at (4,-3) {{$\cdots$}};
\draw[ultra thick,->] (nrp) to[out=90,in=165] (11);
\draw[ultra thick,->] (11) to[out=195,in=75] (nrp);
\draw[ultra thick,->] (nrp) to[out=60,in=140] (21);
\draw[ultra thick,->] (21) to[out=165,in=45] (nrp);
\draw[ultra thick,<-] (nrp) to[out=-60,in=-140] (d1);
\draw[ultra thick,<-] (d1) to[out=-165,in=-45] (nrp);
\draw[ultra thick,<-] (nrp) to[out=-90,in=-165] (k1);
\draw[ultra thick,<-] (k1) to[out=-195,in=-75] (nrp);
\draw[thick,->] (11) to[out=30,in=150] (12);
\draw[thick,->] (12) to[out=30,in=150] (13);
\draw[thick,->] (13) to[out=30,in=150] (1d);
\draw[thick,->] (21) to[out=30,in=150] (22);
\draw[thick,->] (22) to[out=30,in=150] (23);
\draw[thick,->] (23) to[out=30,in=150] (2d);
\draw[thick,->] (k1) to[out=30,in=150] (k2);
\draw[thick,->] (k2) to[out=30,in=150] (k3);
\draw[thick,->] (k3) to[out=30,in=150] (kd);
\draw[thick,->] (12) to[out=-150,in=-30] (11);
\draw[thick,->] (13) to[out=-150,in=-30] (12);
\draw[thick,->] (1d) to[out=-150,in=-30] (13);
\draw[thick,->] (22) to[out=-150,in=-30] (21);
\draw[thick,->] (23) to[out=-150,in=-30] (22);
\draw[thick,->] (2d) to[out=-150,in=-30] (23);
\draw[thick,->] (k2) to[out=-150,in=-30] (k1);
\draw[thick,->] (k3) to[out=-150,in=-30] (k2);
\draw[thick,->] (kd) to[out=-150,in=-30] (k3);
\draw[thick,->] (11) to (21);
\draw[ultra thick,->] (21) to (d1);
\draw[ultra thick,->] (d1) to (k1);
\draw[thick,->] (12) to (22);
\draw[ultra thick,->] (22) to (d2);
\draw[ultra thick,->] (d2) to (k2);
\draw[thick,->] (13) to (23);
\draw[ultra thick,->] (23) to (d3);
\draw[ultra thick,->] (d3) to (k3);
\draw[ultra thick,->] (11) to[out=-60,in=60] (d1);
\draw[thick,->] (11) to[out=-45,in=45] (k1);
\draw[thick,->] (21) to[out=-60,in=60] (k1);
\draw[ultra thick,->] (12) to[out=-60,in=60] (d2);
\draw[thick,->] (12) to[out=-45,in=45] (k2);
\draw[thick,->] (22) to[out=-60,in=60] (k2);
\draw[ultra thick,->] (13) to[out=-60,in=60] (d3);
\draw[thick,->] (13) to[out=-45,in=45] (k3);
\draw[thick,->] (23) to[out=-60,in=60] (k3);
\foreach \x in {1,2,3} {
\node[left] at (\x,-.5) {{\tiny $\alp010$}};
\node at (\x+.45,.3) {{\footnotesize $\lambda_0$}};
\node at (\x+.45,-.7) {{\footnotesize $\lambda_1$}};
\node at (\x+.45,-2.7) {{\footnotesize $\lambda_M$}};
\node at (\x+.45,-.3) {{\footnotesize $\mu_0$}};
\node at (\x+.45,-1.3) {{\footnotesize $\mu_1$}};
\node at (\x+.45,-3.3) {{\footnotesize $\mu_M$}};
\node at (\x+.08,-2.19) {{\tiny $\alp1{M\mbox{-}1}0$}};
\node at (\x+.53,-1.55) {{\tiny $\alp0M0$}};
}
\draw[thick, decorate, decoration={brace,amplitude=10pt}] (4.1,-3.5) -- (.75,-3.5)  node[midway,yshift=-20] {{\large $\mathcal R$}};
\end{tikzpicture}

%% file: detail.tikz.tex
\begin{tikzpicture}[scale=2.6]
%
\node[circle, thick, name=e1, minimum size=.2in] at (1,.5) {{$\vdots$}};
\node[circle, thick, name=e2, minimum size=.2in] at (2,.5) {{$\vdots$}};
\node[circle, thick, name=e3, minimum size=.2in] at (3,.5) {{$\vdots$}};
\node[circle, thick, name=1e, minimum size=.2in] at (.5,0) {{$\cdots$}};
\node[draw, circle, thick, name=11, minimum size=.6in] at (1,0) {{\scriptsize ($m$,$j$-1)}};
\node[draw, circle, ultra thick, name=12, minimum size=.6in] at (2,0) {{\scriptsize $(m,j)$}};
\node[draw, circle, thick, name=13, minimum size=.6in] at (3,0) {{\scriptsize ($m$,$j$+1)}};
\node[circle, thick, name=1d, minimum size=.2in] at (3.5,0) {{$\cdots$}};
\node[circle, thick, name=2e, minimum size=.2in] at (.5,-1) {{$\cdots$}};
\node[draw, circle, thick, name=21, minimum size=.6in] at (1,-1) {{\tiny ($m$+1,$j$-1)}};
\node[draw, circle, thick, name=22, minimum size=.6in] at (2,-1) {{\scriptsize ($m$+1,$j$)}};
\node[draw, circle, thick, name=23, minimum size=.6in] at (3,-1) {{\tiny ($m$+1,$j$+1)}};
\node[circle, thick, name=2d, minimum size=.2in] at (3.5,-1) {{$\cdots$}};
\node[circle, thick, name=3e, minimum size=.2in] at (.5,-2) {{$\cdots$}};
\node[draw, circle, thick, name=31, minimum size=.6in] at (1,-2) {{\tiny ($m$+2,$j$-1)}};
\node[draw, circle, thick, name=32, minimum size=.6in] at (2,-2) {{\scriptsize ($m$+2,$j$)}};
\node[draw, circle, thick, name=33, minimum size=.6in] at (3,-2) {{\tiny ($m$+2,$j$+1)}};
\node[circle, thick, name=3d, minimum size=.2in] at (3.5,-2) {{$\cdots$}};
\node[circle, thick, name=d1, minimum size=.2in] at (1,-2.4) {{$\vdots$}};
\node[circle, thick, name=d2, minimum size=.2in] at (2,-2.4) {{$\vdots$}};
\node[circle, thick, name=d3, minimum size=.2in] at (3,-2.4) {{$\vdots$}};
%
%
\draw[thick,->] (12) to[out=30,in=150] (13);
\draw[thick,->] (12) to[out=-150,in=-30] (11);
\draw[thick,->] (12) to (22);
\draw[thick,->] (12) to[out=-60,in=60] (32);
\draw[thick,->] (12) to[out=-50,in=120] (33);
\draw[thick,->] (12) to[out=-130,in=60] (31);
\draw[thick,->] (12) to[out=-45,in=135] (23);
\draw[thick,->] (12) to[out=-135,in=45] (21);
\node at (2.45,.3) {{\footnotesize $\lambda_m$}};
\node at (1.45,-.3) {{\footnotesize $\mu_m$}};
\node[left] at (2.04,-.58) {{\tiny $\alp m10$}};
\node[left] at (2.28,-1.5) {{\tiny $\alp m20$}};
\node[left] at (1.42,-.58) {{\tiny $\alp m1{-1}$}};
\node[right] at (2.58,-.58) {{\tiny $\alp m1{1}$}};
\node[left] at (1.28,-1.5) {{\tiny $\alp m2{-1}$}};
\node[right] at (2.72,-1.5) {{\tiny $\alp m2{1}$}};
%
\end{tikzpicture}

%% file: sleep.tikz.tex
\begin{tikzpicture}[scale=2.3]
\node[above] at (1.87,.5) { {number of jobs, $j$}};
\draw[ultra thick,|->] (-.26,.5) -- (4,.5);
\node[left] at (-.4,0) {\textbf{off}};
\node[left] at (-.4,-1) {\textbf{sleep}};
\node[left] at (-.4,-2) {\textbf{on}};
\node[draw, circle, thick, name=10, minimum size=.6in] at (0,0) {{\small $(0,0)$}};
\node[draw, circle, thick, name=11, minimum size=.6in] at (1,0) {{\small $(0,1)$}};
\node[draw, circle, thick, name=12, minimum size=.6in] at (2,0) {{\small $(0,2)$}};
\node[draw, circle, thick, name=13, minimum size=.6in] at (3,0) {{\small $(0,3)$}};
\node[circle, thick, name=1d, minimum size=.6in] at (4,0) {{$\cdots$}};
\node[draw, circle, thick, name=20, minimum size=.6in] at (0,-1) {{\small $(1,0)$}};
\node[draw, circle, thick, name=21, minimum size=.6in] at (1,-1) {{\small $(1,1)$}};
\node[draw, circle, thick, name=22, minimum size=.6in] at (2,-1) {{\small $(1,2)$}};
\node[draw, circle, thick, name=23, minimum size=.6in] at (3,-1) {{\small $(1,3)$}};
\node[circle, thick, name=2d, minimum size=.6in] at (4,-1) {{$\cdots$}};
\node[draw, circle, thick, name=30, minimum size=.6in] at (0,-2) {{\small $(2,0)$}};
\node[draw, circle, thick, name=31, minimum size=.6in] at (1,-2) {{\small $(2,1)$}};
\node[draw, circle, thick, name=32, minimum size=.6in] at (2,-2) {{\small $(2,2)$}};
\node[draw, circle, thick, name=33, minimum size=.6in] at (3,-2) {{\small $(2,3)$}};
\node[circle, thick, name=3d, minimum size=.6in] at (4,-2) {{$\cdots$}};
\draw[thick,->] (10) to[out=30,in=150] (11);
\draw[thick,->] (11) to[out=30,in=150] (12);
\draw[thick,->] (12) to[out=30,in=150] (13);
\draw[thick,->] (13) to[out=30,in=150] (1d);
\draw[thick,->] (20) to[out=30,in=150] (21);
\draw[thick,->] (21) to[out=30,in=150] (22);
\draw[thick,->] (22) to[out=30,in=150] (23);
\draw[thick,->] (23) to[out=30,in=150] (2d);
\draw[thick,->] (30) to[out=30,in=150] (31);
\draw[thick,->] (31) to[out=30,in=150] (32);
\draw[thick,->] (32) to[out=30,in=150] (33);
\draw[thick,->] (33) to[out=30,in=150] (3d);
\draw[thick,->] (31) to[out=-150,in=-30] (30);
\draw[thick,->] (32) to[out=-150,in=-30] (31);
\draw[thick,->] (33) to[out=-150,in=-30] (32);
\draw[thick,->] (3d) to[out=-150,in=-30] (33);
\draw[thick,->] (21) to (31);
\draw[thick,->] (22) to (32);
\draw[thick,->] (23) to (33);
\draw[thick,->] (30) to (20);
\draw[thick,->] (20) to (10);
\draw[thick,->] (11) to[out=-60,in=60] (31);
\draw[thick,->] (12) to[out=-60,in=60] (32);
\draw[thick,->] (13) to[out=-60,in=60] (33);
\foreach \x in {0,1,2,3} {
\node at (\x+.45,.3) {{\footnotesize $\lambda$}};
\node at (\x+.45,-.7) {{\footnotesize $\lambda$}};
\node at (\x+.45,-1.7) {{\footnotesize $\lambda$}};
\node at (\x+.45,-2.3) {{\footnotesize $\mu$}};
}
\foreach \x in {1,2,3} {
\node[left] at (\x+.25,-.5) {{\footnotesize $\gamma$}};
\node[left] at (\x,-1.5) {{\footnotesize $\delta$}};
}
\node[left] at (0,-.5) {{\footnotesize $\beta$}};
\node[left] at (0,-1.5) {{\footnotesize $\beta$}};
\end{tikzpicture}

%% file: fatigue.tikz.tex
\begin{tikzpicture}[scale=2.3]
\node[above] at (1.87,.5) { {number of customers, $j$}};
\draw[ultra thick,|->] (-.26,.5) -- (4,.5);
\node[left] at (-.6,0.08) {\textbf{full}};
\node[left] at (-.6,-0.08) {\textbf{speed}};
\node[left] at (-.6,-.92) {\textbf{reduced}};
\node[left] at (-.6,-1.08) {\textbf{speed}};
\node[left] at (-.6,-1.92) {\textbf{slow}};
\node[left] at (-.6,-2.08) {\textbf{speed}};
\node[draw, circle, thick, name=10, minimum size=.6in] at (0,0) {{\small $(0,0)$}};
\node[draw, circle, thick, name=11, minimum size=.6in] at (1,0) {{\small $(0,1)$}};
\node[draw, circle, thick, name=12, minimum size=.6in] at (2,0) {{\small $(0,2)$}};
\node[draw, circle, thick, name=13, minimum size=.6in] at (3,0) {{\small $(0,3)$}};
\node[circle, thick, name=1d, minimum size=.6in] at (4,0) {{$\cdots$}};
\node[draw, circle, thick, name=20, minimum size=.6in] at (0,-1) {{\small $(1,0)$}};
\node[draw, circle, thick, name=21, minimum size=.6in] at (1,-1) {{\small $(1,1)$}};
\node[draw, circle, thick, name=22, minimum size=.6in] at (2,-1) {{\small $(1,2)$}};
\node[draw, circle, thick, name=23, minimum size=.6in] at (3,-1) {{\small $(1,3)$}};
\node[circle, thick, name=2d, minimum size=.6in] at (4,-1) {{$\cdots$}};
\node[draw, circle, thick, name=30, minimum size=.6in] at (0,-2) {{\small $(2,0)$}};
\node[draw, circle, thick, name=31, minimum size=.6in] at (1,-2) {{\small $(2,1)$}};
\node[draw, circle, thick, name=32, minimum size=.6in] at (2,-2) {{\small $(2,2)$}};
\node[draw, circle, thick, name=33, minimum size=.6in] at (3,-2) {{\small $(2,3)$}};
\node[circle, thick, name=3d, minimum size=.6in] at (4,-2) {{$\cdots$}};
\draw[thick,->] (10) to[out=30,in=150] (11);
\draw[thick,->] (11) to[out=30,in=150] (12);
\draw[thick,->] (12) to[out=30,in=150] (13);
\draw[thick,->] (13) to[out=30,in=150] (1d);
\draw[thick,->] (20) to[out=30,in=150] (21);
\draw[thick,->] (21) to[out=30,in=150] (22);
\draw[thick,->] (22) to[out=30,in=150] (23);
\draw[thick,->] (23) to[out=30,in=150] (2d);
\draw[thick,->] (11) to[out=-150,in=-30] (10);
\draw[thick,->] (12) to[out=-150,in=-30] (11);
\draw[thick,->] (13) to[out=-150,in=-30] (12);
\draw[thick,->] (1d) to[out=-150,in=-30] (13);
\draw[thick,->] (21) to[out=-150,in=-30] (20);
\draw[thick,->] (22) to[out=-150,in=-30] (21);
\draw[thick,->] (23) to[out=-150,in=-30] (22);
\draw[thick,->] (2d) to[out=-150,in=-30] (23);
\draw[thick,->] (31) to[out=-150,in=-30] (30);
\draw[thick,->] (32) to[out=-150,in=-30] (31);
\draw[thick,->] (33) to[out=-150,in=-30] (32);
\draw[thick,->] (3d) to[out=-150,in=-30] (33);
\draw[thick,->] (10) to (20);
\draw[thick,->] (11) to (21);
\draw[thick,->] (12) to (22);
\draw[thick,->] (13) to (23);
\draw[thick,->] (20) to (30);
\draw[thick,->] (21) to (31);
\draw[thick,->] (22) to (32);
\draw[thick,->] (23) to (33);
\draw[thick,->] (30) to[out=120,in=-120] (10);

\foreach \x in {0,1,2,3} {
\node at (\x+.45,.3) {{\footnotesize $\lambda$}};
\node at (\x+.45,-.7) {{\footnotesize $\lambda$}};
\node at (\x+.45,-.3) {{\footnotesize $\mu_F$}};
\node at (\x+.45,-1.3) {{\footnotesize $\mu_S$}};
\node at (\x+.45,-2.3) {{\footnotesize $\mu_R$}};
\node[left] at (\x,-.5) {{\footnotesize $\gamma$}};
\node[left] at (\x,-1.5) {{\footnotesize $\delta$}};
}
\node[left] at (-.3,-1.5) {{\footnotesize $\beta$}};
\end{tikzpicture}

%% file: detection.tikz.tex
\begin{tikzpicture}[scale=2.3]
\node[above] at (1.87,.5) { {number of jobs, $j$}};
\draw[ultra thick,|->] (-.26,.5) -- (4,.5);
\node[left] at (-.5,0) {\textbf{uninfected}};
\node[left] at (-.5,-0.92) {\textbf{undetected}};
\node[left] at (-.5,-1.08) {\textbf{infection}};
\node[left] at (-.5,-1.92) {\textbf{detected}};
\node[left] at (-.5,-2.08) {\textbf{infection}};
\node[draw, circle, thick, name=10, minimum size=.6in] at (0,0) {{\small $(0,0)$}};
\node[draw, circle, thick, name=11, minimum size=.6in] at (1,0) {{\small $(0,1)$}};
\node[draw, circle, thick, name=12, minimum size=.6in] at (2,0) {{\small $(0,2)$}};
\node[draw, circle, thick, name=13, minimum size=.6in] at (3,0) {{\small $(0,3)$}};
\node[circle, thick, name=1d, minimum size=.6in] at (4,0) {{$\cdots$}};
\node[draw, circle, thick, name=20, minimum size=.6in] at (0,-1) {{\small $(1,0)$}};
\node[draw, circle, thick, name=21, minimum size=.6in] at (1,-1) {{\small $(1,1)$}};
\node[draw, circle, thick, name=22, minimum size=.6in] at (2,-1) {{\small $(1,2)$}};
\node[draw, circle, thick, name=23, minimum size=.6in] at (3,-1) {{\small $(1,3)$}};
\node[circle, thick, name=2d, minimum size=.6in] at (4,-1) {{$\cdots$}};
\node[draw, circle, thick, name=30, minimum size=.6in] at (0,-2) {{\small $(2,0)$}};
\node[draw, circle, thick, name=31, minimum size=.6in] at (1,-2) {{\small $(2,1)$}};
\node[draw, circle, thick, name=32, minimum size=.6in] at (2,-2) {{\small $(2,2)$}};
\node[draw, circle, thick, name=33, minimum size=.6in] at (3,-2) {{\small $(2,3)$}};
\node[circle, thick, name=3d, minimum size=.6in] at (4,-2) {{$\cdots$}};
\draw[thick,->] (10) to[out=30,in=150] (11);
\draw[thick,->] (11) to[out=30,in=150] (12);
\draw[thick,->] (12) to[out=30,in=150] (13);
\draw[thick,->] (13) to[out=30,in=150] (1d);
\draw[thick,->] (20) to[out=30,in=150] (21);
\draw[thick,->] (21) to[out=30,in=150] (22);
\draw[thick,->] (22) to[out=30,in=150] (23);
\draw[thick,->] (23) to[out=30,in=150] (2d);
\draw[thick,->] (10) to (21);
\draw[thick,->] (11) to (22);
\draw[thick,->] (12) to (23);
\draw[thick,->] (13) to (2d);
\draw[thick,->] (11) to[out=-150,in=-30] (10);
\draw[thick,->] (12) to[out=-150,in=-30] (11);
\draw[thick,->] (13) to[out=-150,in=-30] (12);
\draw[thick,->] (1d) to[out=-150,in=-30] (13);
\draw[thick,->] (21) to[out=-150,in=-30] (20);
\draw[thick,->] (22) to[out=-150,in=-30] (21);
\draw[thick,->] (23) to[out=-150,in=-30] (22);
\draw[thick,->] (2d) to[out=-150,in=-30] (23);
\draw[thick,->] (31) to[out=-150,in=-30] (30);
\draw[thick,->] (32) to[out=-150,in=-30] (31);
\draw[thick,->] (33) to[out=-150,in=-30] (32);
\draw[thick,->] (3d) to[out=-150,in=-30] (33);
\draw[thick,->] (20) to (30);
\draw[thick,->] (21) to (31);
\draw[thick,->] (22) to (32);
\draw[thick,->] (23) to (33);
\draw[thick,->] (30) to[out=120,in=-120] (10);

\foreach \x in {0,1,2,3} {
\node at (\x+.45,.3) {{\footnotesize $\lambda_N$}};
\node at (\x+.45,-.7) {{\footnotesize $\lambda$}};
\node at (\x+.45,-.3) {{\footnotesize $\mu$}};
\node at (\x+.45,-1.3) {{\footnotesize $\mu_I$}};
\node at (\x+.45,-2.3) {{\footnotesize $\mu_I$}};
\node at (\x+.66,-.5) {{\footnotesize $\lambda_V$}};
\node[left] at (\x,-1.5) {{\footnotesize $\gamma$}};
}
\node[left] at (-.3,-1.5) {{\footnotesize $\beta$}};
\end{tikzpicture}

%% file: xyz.tikz.tex
\begin{tikzpicture}[scale=0.4]
\node[below] at (0,6.8) {{\large $S$}};
\node[below] at (0,3.2) {{\large $A$}};
\draw[thick] (0,0) circle (7);
\node at (-5,2) {\textbullet};
\node[left] at (-5,2) {$y$};
\node at (-2,-1) {\textbullet};
\node[below] at (-2,-1) {$z$};
\draw[->,ultra thick] (-5,2) to[out=0,in=90] (-2,-.85);
\node at (-3.3,.25) {{\tiny $q(y,z)$}};
\draw[thick] (0,-1.6) circle (5);
\node at (3,-3) {\textbullet};
\node[right] at (3,-3) {$x$};
\end{tikzpicture}

%% file: clearing.tikz.tex
\begin{tikzpicture}[scale=2.3]
\node[draw, circle, thick, name=0, minimum size=.6in] at (0,0) {{\small $0$}};
\node[draw, circle, thick, name=1, minimum size=.6in] at (1,0) {{\small $1$}};
\node[draw, circle, thick, name=2, minimum size=.6in] at (2,0) {{\small $2$}};
\node[draw, circle, thick, name=3, minimum size=.6in] at (3,0) {{\small $3$}};
\node[circle, thick, name=d, minimum size=.6in] at (4,0) {{$\cdots$}};
\draw[thick,->] (0) to[out=30,in=150] (1);
\draw[thick,->] (1) to[out=30,in=150] (2);
\draw[thick,->] (2) to[out=30,in=150] (3);
\draw[thick,->] (3) to[out=30,in=150] (d);
\draw[thick,->] (1) to[out=-150,in=-30] (0);
\draw[thick,->] (2) to[out=-150,in=-30] (1);
\draw[thick,->] (3) to[out=-150,in=-30] (2);
\draw[thick,->] (d) to[out=-150,in=-30] (3);
\draw[thick,->] (2) to[out=-120,in=-60]  (0);
\draw[thick,->] (3) to[out=-105,in=-75] (0);
\draw[ultra thick,->] (d) to[out=-90,in=-90] (0);
\foreach \x in {1,2,3} {
\node at (\x+.45,.3) {{\footnotesize $\lambda$}};
\node at (\x+.45,-.3) {{\footnotesize $\mu$}};
}
\node at (.45,.3) {{\footnotesize $\lambda$}};
\node at (.45,-.3) {{\footnotesize $\mu+\alpha$}};
\node at (1,-.78) {{\footnotesize $\alpha$}};
\node at (1.5,-1.19) {{\footnotesize $\alpha$}};
\end{tikzpicture}